\documentclass[reqno,12pt]{amsart}
\usepackage{extsizes}
\usepackage{blindtext}

\usepackage{amsmath,amssymb}
\usepackage{graphicx}
\usepackage[top=1in,bottom=1in,left=1in,right=1in]{geometry}
\newtheorem{thm}{Theorem}
\newtheorem{lem}{Lemma}[section]
\newtheorem{prop}[lem]{Proposition}

\theoremstyle{definition}
\numberwithin{equation}{section}

\makeatletter
  \def\th@definition{
  \thm@headfont{\itshape} 
}
\makeatother
\theoremstyle{definition}
\newtheorem{rem}{Remark}[section]

\usepackage{color}
\newcommand{\cb}{\color{black}}

\newcommand{\cred}{\color{black}}

\newcommand{\eps}{\varepsilon}

\newcommand{\be}{\begin{equation} \label}
\newcommand{\ee}{\end{equation}}
\newcommand{\bas}{\begin{eqnarray*}}
\newcommand{\eas}{\end{eqnarray*}}
\newcommand{\R}{\mathbb{R}}
\newcommand{\N}{\mathbb{N}}
\newcommand{\Rn}{\mathbb{R}^n}
\newcommand{\ds}{\displaystyle}
\newcommand{\ts}{\textstyle}

\begin{document}

\title[Universal estimates and Liouville theorems]{Universal estimates and Liouville theorems for \\ superlinear 
problems without scale invariance}

\author{Philippe Souplet}
\address{Universit\'e Sorbonne Paris Nord, CNRS UMR 7539,
Laboratoire Analyse G\'eom\'etrie et Applications,
93430 Villetaneuse, France.}
\email{souplet@math.univ-paris13.fr}

\begin{abstract}
We revisit rescaling methods for nonlinear elliptic and parabolic problems
and show that, by suitable modifications, they may be used for nonlinearities that
are not scale invariant even asymptotically and whose behavior
can be quite far from power like. 

In this enlarged framework, by adapting the doubling-rescaling method from \cite{PQS1, PQS2},
we show that the equivalence found there between universal estimates 
and Liouville theorems remains valid.
In the parabolic case we also prove a Liouville type theorem for a rather large class of 
non scale invariant nonlinearities.
This leads to a number of new results 
for non scale invariant elliptic and parabolic problems,
concerning space or space-time singularity estimates,
initial and final blow-up rates, universal and a priori bounds for global solutions,
and decay rates in space and/or time.

We illustrate our approach by a number of examples, 
which in turn give indication about the optimality of the estimates and of the assumptions.
\end{abstract}

\maketitle

\vskip -6mm
\centerline{\cb\small\it Dedicated to Professor Juan Luis V\'azquez, on the occasion of his 75th Birthday}
\vskip 10mm

\setcounter{tocdepth}{1}

\vspace*{-1cm}
\tableofcontents

\def\eps{\varepsilon}
\def\Rn{{\R}^n}
\def\Rm{{\R}^m}
\def\eps{\varepsilon}

\def\IR {\int_{B_R}}
\def\ISR {\int_{S_R}}
\def\IS {\int_{S^{n-1}}}
\def\Aand{\quad\hbox{ and }\quad}
\def\AAand{\qquad\hbox{ and }\qquad}

\section{Introduction}

We are concerned with Liouville type theorems for elliptic and parabolic problems
and their applications.
Throughout this paper, by a {\it solution} we will always mean a nonnegative solution.
Also solutions will be assumed to be classical unless otherwise specified. 
We shall denote the Sobolev exponent by 
$$p_S=
\begin{cases}
\frac{n+2}{n-2},&\hbox{ if $n\ge 3$} \\
\noalign{\vskip 1mm}
\infty,&\hbox{ if $n\le 2$.}
\end{cases}
$$
Let us first recall the celebrated Gidas-Spruck elliptic Liouville theorem \cite{GSa}
(see also \cite{BVV} for a simplified proof).

\medskip
\noindent {\bf Theorem A.}
{\it Let $1<p<p_S$. Then the equation 
\be{ellup}
-\Delta u=u^p
\ee
has no nontrivial solution in $\R^n$.}
\medskip

\noindent The exponent $p_S$ is critical for nonexistence since, as is well known,
there exist positive solutions
of \eqref{ellup}
in $\R^n$ whenever $n\ge 3$ and $p\ge p_S$ 
(see \cite[Section~9]{QSb} and the references therein; the solution can even be taken bounded and radial).
The question whether the parabolic analogue of Theorem~A is still true remained open for a long time
and was solved in the full subcritical range only recently in \cite{Q21} 
(cf.~also \cite{Q21b} and see \cite{BV98, MS, PQ, PQS2, QSb, Q16}
for previous partial results). Namely:

\medskip
\noindent {\bf Theorem B.}
{\it Let $1<p<p_S$. Then the equation 
\be{parabup}
u_t-\Delta u=u^p
\ee
has no nontrivial solution in $\R^n\times \R$.}
\medskip

Elliptic and parabolic Liouville theorems, in combination with rescaling techniques, have many applications. 
Theorem A (and its half-space analogue) was used in \cite{GSb} to show
a priori bounds and existence for elliptic Dirichlet problems.
In \cite{Gi}, Theorem A was used in combination with energy arguments to prove 
a priori bounds for global solutions of the initial boundary value problem associated with \eqref{parabup},
and next in~\cite{GK}, using also similarity variables, to prove blow-up rates for nonglobal solutions.
Then,  in \cite{PQS1, PQS2}, by combining the rescaling method with a doubling argument,
it was shown that the Liouville property in Theorem~A (resp.~Theorem~B) leads to various
universal estimates for elliptic (resp. parabolic) problems, namely:
space or space-time singularity estimates, initial and final blow-up rates,
 universal and a priori bounds for global solutions,
 decay rates in space and/or time.
 For instance, in the parabolic case, the following was proved in \cite{PQS2}:

\medskip
\noindent {\bf Theorem C.}
{\it Let $p\in (1,p_S)$ be such that\footnote{This was later proved in \cite{Q21} to be true for all $p\in (1,p_S)$.} 
\eqref{parabup}
has no {\cb nontrivial} bounded solution in $\R^n\times \R$.
Let $f:[0,\infty)\to\R$ be a 
continuous function such that
\be{hyplimup}
\lim_{s\to\infty} s^{-p}f(s)=\ell\in (0,\infty).
\end{equation}

(i) For any (possibly unbounded) domain $\Omega\subset\R^n$, $T>0$ and any solution $u$ of
\be{eqf2}
u_t-\Delta u=f(u) \quad\hbox{in $\Omega\times(0,T)$},
 \ee 
we have the estimate\footnote{For the cases $T=\infty$ and $\Omega=\Rn$, 
the last two terms in \eqref{conclthmheatA1} are defined to be $0$.}
\be{conclthmheatA1}
u(x,t)\leq C\bigl(\sigma+t^{-\frac{1}{p-1}}+(T-t)^{-\frac{1}{p-1}}
+{\rm dist}^{-\frac{2}{p-1}}(x,\partial\Omega)\bigr),
\quad x\in\Omega, \quad 0<t<T,
   \end{equation}
where $C=C(n,f)>0$, $\sigma=0$ if $f(s)\equiv \ell s^p$ and $\sigma=1$ otherwise.

\smallskip
(ii) Let $\Omega$ be a uniformly $C^2$ smooth (possibly unbounded) domain of $\R^n$.
For any solution~$u$ of
$$
\begin{cases}
u_t-\Delta u=f(u) \quad&\hbox{in $\Omega\times(0,T)$}, \\
\noalign{\vskip 1mm}
u=0 \quad&\hbox{on $\partial\Omega\times(0,T)$}, 
\end{cases}
$$
we have the estimate
\begin{equation}
u(x,t)\leq C\bigl(1+t^{-\frac{1}{p-1}}+(T-t)^{-\frac{1}{p-1}}\bigr),
\quad x\in\Omega, \quad 0<t<T,
   \label{conclthmheatA2}
   \end{equation}
      where $C=C(n,f,\Omega)>0$.
}

\medskip

Estimates \eqref{conclthmheatA1}, \eqref{conclthmheatA2} are called universal,
 in the sense that the constants $C$ do not depend on the solution.
See \cite{Hu, BV98, CF, FS, MS, QSb} for previous results related with Theorem~C.
Liouville type theorems and their applications have been extended in various directions,
for instance systems or quasilinear problems (see,~e.g.,~\cite{RZ, SZ, S09, QSb, NNPY, Q21} and the references therein).
Concerning Liouville type theorems for semilinear equations with more general nonlinearities, 
let us mention the following important 
generalization of Theorem~A, due to~\cite{LZ}.

\medskip
\noindent {\bf Theorem D.}
{\it Assume that $f\in C([0,\infty))$, with $f(s)>0$ for all $s>0$, and
\be{condLZ}
\hbox{$s^{-p_S}f(s)$ is nonincreasing and nonconstant.}
\ee
 Then the equation
 \be{S-LiouvLZ}
-\Delta u=f(u),\qquad x\in \Rn
\ee 
has no nontrivial solution.}
\medskip

It is natural to conjecture that, assuming an additional superlinearity condition on $f$ at~$\infty$, 
the parabolic analogue of Theorem~D should be true.
However this is still unknown. 
On the other hand it seems that, with the exception of \cite{GIR}\footnote{which provided part of the motivation 
of the present work; see the end of Remark~\ref{RemKnownEll}},
applications of Liouville type theorems to a priori or universal estimates
of nonlinear elliptic and parabolic equations
have up to now required the nonlinearity to satisfy the asymptotic power-like behavior \eqref{hyplimup}.
Also \eqref{conclthmheatA1} with $\sigma=0$ provides temporal decay rates for global solutions of the Cauchy problem 
(taking $T=\infty$ and $\Omega=\Rn$), as well as space-time decay if $\Omega$ is an exterior domain,
but this is available only in the pure power case $f(s)\equiv c s^p$.
For instance, it seems unknown whether some universal estimates are true for 
solutions of the problems
\be{logNL}
-\Delta u=u^p \log^q(2+u)
\quad\hbox{ or }\quad u_t-\Delta u=u^p \log^q(2+u)\qquad (q\in\R\setminus\{0\}),
\ee
or problems with more general factors such as iterated logarithms, etc.

 \smallskip
 
 The first goal of this paper is  to show that, by suitable modifications
of the rescaling and of the doubling-rescaling methods, the above applications to a priori 
estimates and universal singularity or decay estimates
can be developped for elliptic and parabolic problems with much more general nonlinearities,
whose behavior can be quite far from power-like.
\smallskip

A second goal of this paper is to prove some parabolic Liouville type theorems for more general nonlinearities,
providing partial analogues of Theorem~D. 

\smallskip

Additionally, we shall show that elliptic and parabolic Liouville properties and universal estimates possess a kind of
stability with respect to the nonlinearity 
(and this in turn will give rise to further Liouville type theorems).
\smallskip

In a nutshell, our primary motivation is to find out how far one can go with rescaling methods and Liouville type theorems
when dealing with nonlinearities that are {\it not} scale invariant (even asymptotically).
Heuristic considerations relative to rescaling methods (see Section~\ref{SecHeurist}) 
lead to certain relevant classes of functions,
that we will introduce in Section~\ref{SecNot}.
Our results and applications are then presented in Sections~\ref{SecUB1}--\ref{SecStabil}.
The proofs are given in Sections~\ref{SecHeurist}--\ref{SecProofTech}.

\section{Notation and definitions} \label{SecNot}

\subsection{Operators}
We adopt a unified framework which covers elliptic and parabolic problems at the same time.
To this end, for arbitrary (possibly unbounded) given domain $\Omega\subset\Rn$ and $T>0$, we denote either
\be{Lell}
\mathcal{L}=-\Delta,\quad Q=\R^n,\quad D=\Omega,\quad {\cb S=\partial\Omega},\quad X=x,\quad d=d_E
\ee
or
\be{Lparab}
\mathcal{L}=\partial_t-\Delta,\ \ Q=\R^n\times\R,\ \ D=\Omega\times(0,T),\ \ 
{\cb S=\partial\Omega\times(0,T)},\ \  X=(x,t),\ \  d=d_P,
\ee
where $d_E, d_P$ respectively denote the elliptic and parabolic distances, namely 
$$d_E(x,y)=|x-y|,\quad d_P\bigl((x,t),(y,s)\bigr)=|x-y|+|t-s|^{1/2},\quad x,y\in\Rn,\ t,s\in\R.$$
On the other hand, we shall always assume
\be{hypcontpos}
\hbox{$f\in C([0,\infty))$ with
$f(s)>0$ for all $s>0$.}
\ee
By a strong solution of $\mathcal{L}u=f(u)$, we mean a function $u$ which belongs to $W^{2,r}_{loc}$ for each finite~$r$ 
(or $W^{2,1;r}_{loc}$ in the parabolic case) and satisfies the equation a.e. 
Moreover, when boundary conditions are imposed, a strong solution is assumed to be continuous up the boundary.
Of course strong solutions are classical if $f$ is locally H\"older continuous (which we do not assume in general).

\smallskip

\subsection{Nonlinearities}
We introduce two classes of functions. We say that:
\smallskip

$\bullet$ $f$ has {\it regular variation at $\infty$} (resp., $0$) with index $m\in\R$ if 
\be{hypregulvar}
L(s):= s^{-m} f(s) \hbox{ satisfies } 
\lim_{\lambda\to \infty \atop ({\rm resp.,}\, \lambda\to0)} \frac{L(\lambda s)}{L(\lambda)}=1\ \hbox{ for each $s>0$.}
\ee

$\bullet$ $f$ has {\it controled variation} if 
\be{hypcontrolvar}
\inf_{\lambda >0\atop s\in \mathcal{K}} \ \frac{f(\lambda s)}{f(\lambda)} >0
\ \hbox{ for each compact $\mathcal{K}\subset(0,\infty)$.} 
\ee

\noindent  
A positive function $L\in C(0,\infty)$ with the property in \eqref{hypregulvar}
is called a function with {\it slow variation} at $\infty$ (resp., at $0$).
Functions of regular and slow variation were introduced by Karamata~\cite{Ka}
and have been extensively studied. See \cite{Se} for a general reference
and e.g.,~\cite{BGT, Cir} for applications in probability theory and PDE, respectively.
Functions with controled variation (also called R-O varying functions),
were introduced in~\cite{Ka2}.
See Section~\ref{SubSecExVar} for examples, and Section~\ref{SubSecPrelim} for useful properties. 

\medskip

Also we will sometimes use the following properties. We say that:
\smallskip

$\bullet$ $f$ is {\it superlinear} if
\be{hypsuperlin}
\inf_{\lambda>0}\ {f(\lambda s)\over s f(\lambda)}\to \infty, \quad s\to \infty.
\ee

$\bullet$ $f$ is {\it almost decreasing}\footnote{almost decreasing functions
were introduced by S.N.~Bernstein \cite{Ber}.} 
if 
\be{hypmodervar}
\inf_{\lambda >0\atop s\in (0,1]} \ \frac{f(\lambda s)}{f(\lambda)} >0.
\ee

\subsection{Examples of functions with regular or controled variation} \label{SubSecExVar} 
If $L$ has regular variation at $0$ and $\infty$, then it has controled variation. 
Functions with controled (resp., regular)
variation are stable by multiplication,
as well as by raising to any real power.
Typical examples of functions $L$ with slow variation at $0$ and $\infty$,
and having singular or oscillating behaviors at $\infty$, are given~by: 
\be{examplesSlowVar}
\left\{\begin{aligned}
&\hbox{\ $\bullet$\  $\log^a (K+s)$ for $K>1$ and $a\in\R$,} \\
\noalign{\vskip 1mm}
&\hbox{\ $\bullet$\  the iterated logarithms {\cb $\log_m(K+s)$,\footnotemark}} \\
 \noalign{\vskip 1mm}
 &\hbox{\ $\bullet$\  $\exp(\log s/\log(K+|\log s|))$ for $K>1$,} \\
 \noalign{\vskip 1mm}
 &\hbox{\ $\bullet$\  $\exp(|\log s|^\nu)$ with $\nu\in (0, 1)$,} \\
 \noalign{\vskip 1mm}
 &\hbox{\ $\bullet$\  the strongly oscillating functions} \\
 &\hbox{\qquad $\bigl[\log(3+s)\bigr]^{\sin[\log\log(3+s)]}
\quad\hbox{ and }\quad
\exp[|\log s|^\nu\cos(|\log s|^\nu)],\ \ \nu\in (0, 1/2)$,} \\
\noalign{\vskip 1mm}
 &\hbox{\ $\bullet$\  $1 +a \sin\bigl(\log^\nu(2+s)\bigr)$ with $\nu\in (0, 1)$ and $|a|<1$.} \\
\end{aligned}
\right.
\ee
\footnotetext{\cb where $\log_m=\log\circ\dots\circ\log$ ($m$ times), $m\in\N^*$
and $K>[\exp\circ\dots\circ\exp](0)$.}

\noindent 
Similar behaviors at $0$ (or at both $0$ and $\infty$) are obtained by replacing $s$ by $s^{-1}$ or by $s+s^{-1}$.

\smallskip

Any positive function that is bounded and bounded away from zero has
controled variation. A typical example of strongly oscillating function with controled but not regular variation is given by
$$f(s)=s^{\sin[\log\log(3+s+s^{-1})]}.$$

\section{Nonlinearities with regular variation at infinity: \\ universal singularity estimates} \label{SecUB1}
Our first main result provides universal singularity estimates whenever $f$ has subcritical regular variation at infinity.

\goodbreak

\begin{thm} \label{thm1-0}
Assume \eqref{Lell} or \eqref{Lparab}, and \eqref{hypcontpos}.
Suppose that
\be{hypUB1a0}
\hbox{$f$ has regular variation at $\infty$ with index in $(1,p_S)$.}
\ee

(i) There exists a constant $C=C(n,f)>0$ such that, for any 
strong solution $u$ of
\be{eqf}
 \mathcal{L} u=f(u) \quad\hbox{in $D$},
 \ee 
we have the estimate
\be{EstimA-0}
{f(u(X))\over u(X)}\le C\bigl(1+d^{-2}(X,\partial D)\bigr),\quad\hbox{ for all $X\in D_+:=D\cap\{u\ge 1\}$.}
\ee

(ii) Let $\Omega\subset\R^n$ be a uniformly $C^2$ domain and $T>0$.
There exists a constant $C=C(n,f,\Omega)>0$ such that, for any 
strong solution $u$ of \eqref{eqf} with $u=0$ on ${\cb S}$,
we have the estimate
\be{EstimB}
\begin{cases}
u\le C&\hbox{in $D$, in case \eqref{Lell}} \\
\noalign{\vskip 1mm}
\ds\frac{f(u)}{u}\le C\bigl(1+t^{-1}+(T-t)^{-1}\bigr)&\hbox{in $D_+$, in case \eqref{Lparab}.}
\end{cases}
\ee
\end{thm}

\medskip

 \begin{rem}\label{formHardy0}
 (i) We see that Theorem~\ref{thm1-0} applies to any (continuous) nonlinearity of the form
$$\hbox{$f(s)=s^pL(s)$, where $p\in(1,p_S)$ and $L$ has slow variation at $\infty$}.$$
This in particular covers all the functions $L$ in \eqref{examplesSlowVar}.
 Also, the nonnegativity of $f$ could be assumed only for $s$ large.
 
\smallskip

(ii) The form of estimate \eqref{EstimA-0} is natural.
It is reminiscent of the classical condition, of Hardy type, on the potential $V$ for the problem $-\Delta u=V(x)u$.
Of course, we may replace $D_+$ with $D$ if the function $f(s)/s$ is assumed to be bounded near $s=0$.

\smallskip

(iii) The upper estimates in Theorem~\ref{thm1-0} can be seen to be sharp in a number of cases.
For instance, in the parabolic case with $\Omega=\Rn$, this gives the blow-up 
rate\footnote{assuming $\limsup_{s\to 0}f(s)/s<\infty$ 
for simplicity}
\be{BUrate}
\ds\frac{f(u)}{u}\le C(T-t)^{-1}
\ee
and, under the assumptions of Theorem~\ref{thm1-0} (or Theorem~\ref{thm1}),
this coincides, up to a multiplicative constant, 
with the exact rate of the special ODE solutions 
(see Proposition~\ref{propSharp}). 
 \end{rem}
 
\smallskip

 \begin{rem}
{\bf Known results on blow-up for non scale invariant parabolic problems.} 

\smallskip

(i) The logarithmic equation 
\be{eqHZ}
u_t-\Delta u=|u|^{p-1}u\log^q(2+u^2),\quad x\in\Rn,\ t>0
\ee
  was recently studied in \cite{DNZ, HZ2}.
It is proved in \cite{HZ2} that, for $p\in(1,p_S)$ and $q\in\R$, 
any blowup solution of \eqref{eqHZ} with initial data in $L^\infty(\Rn)$
satisfies $|\frac{f(u)}{u}|\le C(T-t)^{-1}$ for some constant $C$ depending on the solution.
Note that, for nonnegative solutions, Theorem~\ref{thm1-0} contains this result as a special case
and that, unlike in \cite{HZ2}, the constant in \eqref{EstimB} is universal. 
However our approach does not apply to sign changing solutions.
The method in \cite{HZ2}, completely different from ours, relies on an explicit similarity change of variables
adapted to the special form of the nonlinearity, combined with suitable interpolation estimates.
See also \cite{DNZ} where a class of solutions of \eqref{eqHZ} with prescribed single-point blow-up profile is constructed,
and \cite{HZ} for related results on the corresponding nonlinear wave equation.

\smallskip

(ii) For \eqref{eqf2} in a bounded domain with rather general nonlinearities,
 but for the strongly restricted class of time increasing solutions,
we note that estimate \eqref{BUrate} follows from \cite[Section~4]{FML85}
(based on maximum principle methods).
On the other hand, under suitable assumptions on the non scale invariant nonlinearity $f$, 
results on blow-up sets 
for equation \eqref{eqf2} where obtained in \cite{Fuj} 
by means an appropriate nonlinear change of variables (``quasi-scaling'' transformation).
\end{rem}

\begin{rem} \label{RemKnownEll}
{\bf Known results on non scale invariant elliptic problems.} 
We refer to \cite{CirDu, Cir, GIR, CRT} for previous studies of nonlinear elliptic problems
 involving the class of regularly varying nonlinearities. 
The works \cite{CirDu, Cir} are devoted to the classification of isolated singularities for problems 
with absorption terms (i.e.,~with $\Delta$ instead of $-\Delta$), 
and are based on sub-/supersolution methods.
In \cite{CRT}, existence results are proved for a bifurcation problem with source. 
The methods there are different and rely on variational arguments.
Closer to our topic is the work \cite{GIR}, where a priori estimates are proved in the case of regularly varying, slightly superlinear nonlinearities with sign changing coefficients (typically of the form $-\Delta u=a(x)u\log^q(1+u)$ with $q>0$).
The rescaling argument used in \cite{GIR} provided us a clue to the more general arguments
that we develop here.
\end{rem}
 
\section{Other results on universal estimates}  \label{SecUB2}
 
\subsection{Nonlinearities with controled variation 
and decay estimates}
As an advantage, Theorem~\ref{thm1-0} has simple assumptions and covers the whole Sobolev subcritical range.
As a drawback, it only allows slow (e.g.~logarithmic) factors or oscillations around a given power.
Moreover, it does not allow to measure the possible decay of solutions as $|x|$ or $|t|\to\infty$,
since no assumptions have been made on the behavior of $f$ at $0$.
These two limitations are partially relaxed in the next result, under different sets of
 assumptions on $f$.
To this end we introduce the exponents
$$p_{sg}:=\frac{n}{(n-2)_+}\le p_S, \quad p_F:=\frac{n+2}{n}<p_B:=\frac{n(n+2)}{(n-1)^2},$$
and set
 $$\begin{cases}
\ p_c=p_S,\quad  p_*=p_{sg}, \quad 
n_*=2\quad&\hbox{in case \eqref{Lell}}, \\
\noalign{\vskip 1mm}
\ p_c=p_B,\quad p_*=p_F,\quad 
n_*=1 \quad&\hbox{in case \eqref{Lparab}}. \\
\end{cases}
$$
 
\smallskip

\goodbreak

\begin{thm} \label{thm1}
Assume \eqref{Lell} or \eqref{Lparab}, and \eqref{hypcontpos}.
Suppose either
\be{hypUB1c}
\left\{\begin{aligned}
&\ \hbox{$f(s)=s^pL^a(s)$ where $L$ has controled variation, $p\in(1,p_S)$,} \\
&\ \hbox{and $a>0$ is sufficiently small (depending on $n, p, L$),}
\end{aligned}
\right.
\ee
or
\be{hypUB1a}
\left\{\begin{aligned}
&\ n\le n_*\quad\hbox{or}\quad \hbox{ $f(s)/s^p$ nonincreasing for $s>0$ and some $p\in (1,p_c)$,} \\
&\ \hbox{$f$ has regular variation at $0$ and $\infty$ with indices in $(1,p_S)$,}
\end{aligned}
\right.
\ee
or
\be{hypUB1b}
\left\{\begin{aligned}
&\ \hbox{$f(s)/s^p$ is almost decreasing for some $p\in(1,p_*]$,} \\
&\ \hbox{$f$ has controled variation and is superlinear.}
\end{aligned}
\right.
\ee
Then, for any strong solution $u$ of \eqref{eqf};
we have the estimate
\be{EstimA}
{f(u(X))\over u(X)}\le {C\over d^2(X,\partial D)},\quad\hbox{ for all $X\in D$,}
\ee
where $C=C(n,f)>0$ in cases \eqref{hypUB1a}, \eqref{hypUB1b} and $C=C(n,p,L)>0$ in case \eqref{hypUB1c}.
\end{thm}

We will see (cf.~the examples in the next subsection)
that conditions \eqref{hypUB1c} and \eqref{hypUB1b} allow nonlinearities that 
can be very far from any precise power behavior,
with infinite oscillations between two different powers.

\begin{rem} \label{remDecay}
(i) {\bf Sharpness of the assumptions.} 
Observe that the ``homogeneous'' estimate \eqref{EstimA} implies\footnote{taking $D=B_R$ or $D=B_R\times(-R^2,R^2)$ and 
{\cb letting} 
$R\to\infty$.} 
the Liouville property in $Q$.
Consequently \eqref{EstimA} in particular fails whenever
\eqref{S-LiouvLZ} admits a positive solution.
In this respect we have the following useful counter-example.
Let
$${\cb f}(s)=\bigl[A+B\min(s^{p-1},1)\bigr] s^p,\quad A,B>0,
\quad \hbox{ with $n\ge 3$ and $p>n/(n-2)$.}$$
Direct computation shows that, for suitable $A,B>0$, the function $v(x)=(1+|x|^2)^{-1/(p-1)}$ 
is a solution of \eqref{S-LiouvLZ}.

It follows that the smallness assumption on $a$ in \eqref{hypUB1c} cannot be relaxed
(noting that $L(s)=A+B\min(s^{p-1},1)$ has controled variation).
Likewise, for $n\ge 3$ and $p>n/(n-2)$, condition \eqref{hypUB1b}  is no longer sufficient.
Also, the monotonicity condition in \eqref{hypUB1a} is not technical.
In the parabolic case, we refer to \cite{Tal09, Kur} and \cite[Proposition~21.14]{QSb} for related counter-examples with $p>p_F$.

\smallskip

(ii) {\bf Decay of bounded solutions.}
If one considers bounded solutions of \eqref{eqf} in an unbounded domain $D$
and is only interested in the decay rates as $|x|$ and/or $|t|\to\infty$, 
then the assumptions on $f$ for large $s$ (namely $s\ge M=\|u\|_\infty$) are not needed,
and estimate \eqref{EstimA} remains true (with a constant $C$ depending also on $M$).\footnote{Indeed 
it suffices to redefine the nonlinearity by $\tilde f(s)=f(M)(s/M)^q$ for $s>M$ with $q>1$ close to $1$,
and to apply Theorem~\ref{thm1} to this modified~$f$.}
\end{rem}

We also have {\cb an} analogue of Theorem~\ref{thm1} for boundary value problems,
for which we define the additional exponent
$p_{**}=\frac{n+1}{n-1}$ in case \eqref{Lell} and 
$p_{**}=1+\frac{2}{n+1}$ in case \eqref{Lparab}.

\begin{thm} \label{thm1b}
Let $\mathcal{L}, f$ be as in Theorem~\ref{thm1}
{\cb under either assumption \eqref{hypUB1c}, or \eqref{hypUB1b} with} $p_*$ replaced by $p_{**}$.
Let $\Omega\subset\R^n$ be a uniformly $C^2$ domain and $T>0$.
There exists a constant $C=C(n,f,\Omega)>0$ such that any 
strong solution $u$ of \eqref{eqf} with $u=0$ on ${\cb S}$ satisfies \eqref{EstimB}.
\end{thm}

\begin{rem} \label{remOther}
{\bf Known results in the elliptic lower range.} 
In the lower range corresponding to assumption \eqref{hypUB1b}
in the elliptic case, some the estimates
are already known, sometimes under weaker assumptions
or sharper form.
However we have included this in Theorems~\ref{thm1}-\ref{thm1b} 
since the conclusion follows from our unified general approach
and still seems new in some cases.

Namely, under assumption \eqref{hypUB1b} (with $p_{**}=(n+1)/(n-1)$ instead of $p_*$),
estimate \eqref{EstimB} from Theorem~\ref{thm1b} 
has been already known by other methods, except for the critical case $p=p_{**}$.
More precisely for the elliptic Dirichlet problem in bounded domains, 
the a priori estimate $u\le C$ holds for any continuous $f$ such that
$\lim_{s\to\infty} f(s)/s=\infty$ and $\lim_{s\to\infty} f(s)/s^{p_{**}}=0$.
See \cite{BT}, \cite{BVV, QS04} and \cite{Sir}, where the methods are respectively based 
on Hardy-Sobolev inequalities, on $L^p_\delta$ spaces, and on Harnack inequalities.
These works also treat some other problems (gradient nonlinearities, systems, 
 fully nonlinear equations, ...).
 As for the case of isolated singuarities (i.e.,~Theorem~\ref{thm1} with $\Omega=B_1\setminus\{0\}$), 
 it is shown in~\cite{Tal0} that any solution of
 $0\le -\Delta u\le u^{n/(n-2)}$ satisfies the less singular estimate $u\le O(|x|^{2-n})$ as $|x|\to 0$.
\end{rem}

\subsection{Examples for Theorems~\ref{thm1} and \ref{thm1b}} \label{SubSecEx}

The conclusions of Theorem~\ref{thm1} 
hold for the functions $f$ in the following examples
(see Section~\ref{SecProofTech} for details).

\smallskip

 {\it Example 1.} 
$f(s)=s^p\log^q(K+s),$
where $p, q$ satisfy either
\begin{itemize}
\item[(i)] $1<p<p_*$, $q\in \R$ and $K>1$, or

\item[(ii)] $p_*\le p<p_c$ and either:
\begin{itemize}
\item[(ii1)] $q<p_c-p$ and $K\ge 1$;
\item[(ii2)] $q=p_c-p$ and $K>1$;
\item[(ii3)] $q>p_c-p$ and $K> {p_c-p\over q}\exp\bigl[{q\over p_c-p}-1\bigr] \ (>1)$. 
\end{itemize}

\item[(iii)] (in the parabolic case) $p_B\le p<p_S$ and $|q|$ small.
\end{itemize}
Moreover, these conditions are not technical.
Indeed, if $p,q,K$ satisfy either 
\be{Ex1a2}
1<p<p_S,\ \ q>p_S-p,\ \ K=1\quad\hbox{or}\quad p\ge p_S,\ \ q>0,\ \ K\ge 1,
\ee
then there exists a positive solution $v$ of $-\Delta v=v^p\log^q(K+v)$ in $\R^n$ (see Proposition~\ref{lemEx1b})
and consequently, estimate \eqref{EstimA} fails in that case
(cf.~Remark~\ref{remDecay}(i)).

\medskip

{\it Example 2.} Our ``wildest'' typical example in terms of oscillations 
(with optimal range of~$p$) is given by:
\be{Ex1a}
f(s)=s^{p+a\sin[\log\log(3+s+s^{-1})]}
\ee
with $p\in(1,p_S)$ and $a\in(0,p-1)$ small (depending on $p, n$).
This nonlinearity is very far from any precise power behavior, with
infinite oscillations between the powers $p-a$ and $p+a$ both at $\infty$ and at $0$.
The only drawback is that the size of $a$ is not quantitatively estimated.

\medskip

{\it Example 3.} 
$$\hbox{$f(s)=s^p+s^q$ \quad or \quad}
f(s)=\begin{cases}
s^p,&\hbox{ if $s\in[0,1]$} \\
\noalign{\vskip 1mm}
s^q,&\hbox{ if $s\ge 1$,}
\end{cases}
\quad\hbox{ with $1<p,q<p_c$. }
$$
This provides examples with different power rates for singularities and for decay in \eqref{EstimA}.

\smallskip

{\it Example 4.}  Any $f$ satisfying \eqref{hypcontpos} and such that
\be{Ex1}
\hbox{$f(s)/s^{p_*}$ is nonincreasing and $f(s)/s^{\ell}$ is nondecreasing for some $\ell>1$}
\ee
(with $p_*$ replaced by any finite $p>1$ if $p_*=\infty$).
As in the previous example, this allows nonlinearities that are strongly oscillatory.
Namely, for any $1<m<p<p_*$ there exist such nonlinearities which satisfy
$f(s_i)=s_i^p$ and $f(t_i)=t_i^m$ for some sequences $s_i,t_i\to\infty$
(see Lemma~\ref{lemEx4}).
Note that such $f$ (as well as that in \eqref{Ex1a}) do not have regular variation.
 One can find similar examples concerning the behavior of $f$ at $0$ (or even at both $0$ and $\infty$).

\smallskip

{\it Example 5.} Any $f\in C([0,\infty))$ such that
\be{Ex2}
c_1 s^p\le f(s)\le c_2 s^p,
\quad \hbox{with $1<p\le p_*$ and $c_1, c_2>0$.}
\ee
We see from the counter-example in Remark~\ref{remDecay}(i) that condition \eqref{Ex2}  is no longer sufficient
for $n\ge 3$ and $p>n/(n-2)$.
We also note that for each $p\in (1,p_S)$,
\eqref{EstimA} was known to be true if \eqref{Ex2} is satisfied with $c_2/c_1$ close enough to $1$.
See \cite{Tal} and \cite[Remarks 8.5 and 8.8]{QSb} for the elliptic case, and 
\cite[Remark~26.11a]{QSb} (in combination with Theorem~B) for the parabolic case.

\medskip

{\cb The conclusions of Theorem~\ref{thm1b} also hold for $f$ as in Example~2, or as in Examples~4-5
with $p_*$ replaced by $p_{**}$.}

\subsection{From Liouville type theorems to universal estimates}

 The following result shows that, under the regular variation assumptions,
the Liouville property guarantees the validity of universal estimates. 
It will be used in the proof of Theorem~\ref{thm1} in case \eqref{hypUB1a}.

\begin{thm} \label{thm1c}
Assume \eqref{Lell} or \eqref{Lparab}, and \eqref{hypcontpos}.
Suppose that
\be{HypLiouville}
\hbox{$\mathcal{L} v=f(v)$
has no nontrivial bounded strong solution in $Q$}
\ee
and
\be{HypLiouvilleReg}
\hbox{$f$ has regular variation at $0$ and $\infty$ with indices in $(1,p_S)$.}
\ee
There exists a constant $C=C(n,f)>0$ such that, for any domain $D\subset Q$ and any 
solution $u$ of
\eqref{eqf} we have the estimate \eqref{EstimA}.

\end{thm}

\begin{rem}
Theorem~\ref{thm1c} shows that,
as already noted in \cite{PQS1, PQS2} for the special case $f(s)=s^p$,
the Liouville type property \eqref{HypLiouville} is 
actually {\it equivalent} to the validity of universal estimates.
Indeed estimate \eqref{EstimA} implies that equation \eqref{eqf} with $D=Q$
does not admit any nontrivial solution
(cf.~Remark~\ref{remDecay}(i)).
Note that, through \eqref{EstimA}, property \eqref{HypLiouville} in turn implies nonexistence of unbounded entire solutions.
\end{rem}

By suitable modifications of {\cb the} proofs 
(cf.~\cite{PQS1,PQS2}),
similar results could be extended to, e.g., systems or quasilinear problems.
We leave this to the interested readers.

\section{Parabolic Liouville type theorems without scale invariance}  \label{SecResLiouvParab}

As mentioned above, it is unknown whether the parabolic analogue of Theorem~D is true.
The proof of Theorem~D is based on moving planes arguments which do not seem 
to carry over to the parabolic case. As for the delicate proof \cite{Q21} of Theorem~B, 
it makes a crucial use of the similarity change of variables and seems
 to require the exact form of the pure power nonlinearity.
We shall here give a partial positive answer, by adapting the proof in \cite{BV98}
which is a parabolic extension of the method in~\cite{GSa, BVV}, based on 
 integral estimates combined with Bochner's formula.
Thus consider the equation
\be{S-LiouvBVeqnf}
u_t-\Delta u=f(u),\qquad x\in \Rn,\quad t\in \R.
\ee 

 As typical cases, we have the following result.

\begin{thm} \label{S-thmLiouvJLV2}
Assume \eqref{hypcontpos} and
\be{S-LiouvParabHyp1}
n=1\quad\hbox{or}\quad \hbox{ $f(s)/s^p$ nonincreasing for $s>0$ and some $p\in (1,p_B)$.} 
\ee
In addition, suppose that either
\begin{eqnarray}
\hskip-17mm&&\hbox{$f$ has regular variation at $0$ and $\infty$ with indices $>1$; or} \label{S-LiouvParabHyp1b} \\
\noalign{\vskip 1mm}
\hskip-17mm&&\hbox{$f$ is nondecreasing near $\infty$, $f(s)\ge s^q$ at $\infty$ and $f(s)/s^\sigma$ is nondecreasing near $0$,}
 \label{S-LiouvParabHyp3}
\end{eqnarray}
 for some $q>1$, $\sigma>0$.
Then equation \eqref{S-LiouvBVeqnf}
has no nontrivial 
strong solution.
\end{thm}

The result applies for instance to the case
$$f(s)=s^p\log^q(K+s)$$
with $1<p<p_B$ and $q,K$ as in assumption (ii) of Example 1 in Section~\ref{SubSecEx}.
Theorem~\ref{S-thmLiouvJLV2} is a special case of the following more general result
(but whose formulation is a bit less transparent).
Denote $\omega_1=(0,1]$, $\omega_2=(1,\infty)$ and set
\be{S-LiouvParabHypInt}
\tilde f(s)=\int_0^s z^{-1}f(z)\,dz
\quad\hbox{ (assuming }\int_0^1 z^{-1}f(z)\,ds<\infty).
\ee

\begin{thm} \label{S-thmLiouvJLV1}
Assume \eqref{hypcontpos}, \eqref{S-LiouvParabHypInt} and
\be{S-LiouvParabHypBV}
n=1\quad\hbox{ or }\quad f(s)\le p\tilde f(s),\ s>0, \ \hbox{ for some $p\in (1,p_B)$.}
\ee
In addition, for some ${2\over 3}<m_2<m_1<m^*:={2(n+2)\over 3n+2}$ and $c>0$,
we assume  
\be{Qbddunbddnonc}
f(s)\ge cs^{m_i} \tilde f^{2m_i-1}(s), 
\quad s\in \omega_i,\quad i\in\{1,2\}.
\ee 
Then equation \eqref{S-LiouvBVeqnf}
has no nontrivial strong solution.
\end{thm}

\begin{rem} (i) {\bf One-dimensional case.}
For $n=1$, a related Liouville type theorem is proved in \cite{P20}. Namely, \eqref{S-LiouvBVeqnf} has no {\it bounded} positive solutions,
whenever $f$ is locally Lipschitz continuous on $[0,\infty)$ and positive for $s>0$.
Theorem~\ref{S-thmLiouvJLV2} for $n=1$ shows that the boundedness and local Lipschitz assumptions can be removed if 
$f$ instead satisfies assumption \eqref{S-LiouvParabHyp1b} or \eqref{S-LiouvParabHyp3}.
The superlinearity assumption at $\infty$ in \eqref{S-LiouvParabHyp3} is more or less optimal since,
as noted in \cite{P20}, (unbounded) positive solutions do exist whenever $f$ is globally Lipschitz,
given by solutions of the ODE $y'(t)=f(y(t))$ with $y(0)>0$.

The proof in \cite{P20} is completely different, based on zero number 
(and the result remains true for sign-changing solutions such that the zero number of $u(\cdot,t)$ stays bounded).

\smallskip

(ii) The nonexistence of nontrivial solutions of \eqref{S-LiouvBVeqnf}
in the special case $f(s)=s^p+s^2$ with $1<p<p_B$ and $n\le 5$ has been proved in \cite{JWY}
(where a class of systems is actually studied), by using the approach from \cite{BV98}. However, we note that the nonlinearity is asymptotic to a power.
\end{rem}

\begin{rem} \label{remLiouvParab}
(i) Assumption \eqref{Qbddunbddnonc} is sharp. Indeed, for $f(s)=s$ (hence $\tilde f(s)=s$),
\eqref{Qbddunbddnonc} with $i=2$ is satisfied only for $m_2\le 2/3$ and \eqref{S-LiouvBVeqnf} does admit a nontrivial entire solution,
given by $u(x,t)=e^t$. And for $f(s)=s^{p_S}$ (hence $\tilde f(s)=cs^{p_S}$),
\eqref{Qbddunbddnonc} with $i=1$ is satisfied only for $m_2\ge m^*$ and \eqref{S-LiouvBVeqnf}
admits a positive stationary solution.
However, we do not know (nor expect) the condition $p<p_B$ in \eqref{S-LiouvParabHypBV} to be sharp (see after Theorem~D).
\smallskip

(ii) It should be possible to take $\sigma=0$ in assumption~\eqref{S-LiouvParabHyp3}, 
by suitable modifications of the proof of Theorem~\ref{S-thmLiouvJLV1}
(taking $q<0$ small instead of $q=0$ in \eqref{S-LiouvGSestIJK0} 
and redefining $\tilde f(s)$  accordingly by $\int_0^s z^{-1-q}f(z)\,dz$).
We have refrained to develop this to avoid additional technical complication
in the statement and proof of Theorem~\ref{S-thmLiouvJLV1}.
\smallskip

(iii) Note that the second part of \eqref{S-LiouvParabHypBV} 
is equivalent to $\tilde f(s)/s^p$ being nonincreasing for $s>0$. Here is an example of nonlinearity
satisfying the assumptions of Theorem~\ref{S-thmLiouvJLV1} but not
assumption \eqref{S-LiouvParabHyp1} of Theorem~\ref{S-thmLiouvJLV2} {\cb for $n\ge 2$}:
\be{counterexftilde}
f(s)=\begin{cases}
s^m+s^q,&\hbox{ if $s\in[0,a]$} \\
\noalign{\vskip 1mm}
(1+a^{q-m})s^m,&\hbox{ if $s\ge a$,}
\end{cases}
\ee
with $1<m<p_B<q$, for suitably chosen $a>0$ (see~Lemma~\ref{lemEx2}).
\end{rem}

\section{Stability of Liouville properties and of universal estimates} \label{SecStabil}

It is known that the Liouville property may be rather sensitive to the nonlinearity.
An assumption on the leading order at $0$ and/or infinity is usually not sufficient (except in the lower range $p\le p_*$, 
cf.~Lemma~\ref{lemGidFuj}) and oscillations in the nonlinearity
may be responsible for the existence of entire solutions.
See for instance the counter-example in Remark~\ref{remDecay}(i).
In this section we shall nevertheless show that 
Liouville properties and universal estimates are stable with respect to the nonlinearity 
in a certain sense.

\smallskip

Consider a family $\mathcal{F}=\{f_\eps,\ \eps\in(0,\eps_{\cb 1})\}$ of continuous 
functions $f_\eps:[0,\infty)\to[0,\infty)$, such that
\be{hypfeps}
f_\eps(s)>0,\quad s>0.
\ee
We assume that the limits below exist and are finite:
\be{hypgeps}
g(s):=\lim_{\eps\to 0}f_\eps(s),\quad\hbox{ uniformly for $s$ in compact subsets of $(0,\infty)$,}
\ee
\be{hypphieps}
\phi(s):= \lim_{\eps, \lambda\to 0}{f_\eps(\lambda s)\over f_\eps(\lambda)},
\quad\hbox{\cb uniformly for $s$ in compact subsets of $(0,\infty)$,}
\ee
\be{hyppsieps}
\psi(s):=\lim_{\eps\to 0,\, \lambda\to \infty}{f_\eps(\lambda s)\over f_\eps(\lambda)},
\quad\hbox{\cb uniformly for $s$ in compact subsets of $(0,\infty)$,}
\ee
We also assume the uniform superlinearity condition:
\be{hypsuperlineps}
\inf_{\eps\in(0,\eps_{\cb 1}),\,\lambda>0}\ {f_\eps(\lambda s)\over sf_\eps(\lambda )}\to \infty, \quad s \to \infty.
\ee
We have the following result
 (which {\cb will imply Theorem~\ref{thm1}} under assumption \eqref{hypUB1c} as a special case).

\begin{thm} \label{thm2eps}
Assume \eqref{Lell} or \eqref{Lparab}, and let \eqref{hypfeps}-\eqref{hypsuperlineps} be satisfed. Suppose in addition that the problem
$$\mathcal{L}w=h(w),\quad x\in Q$$
with $h\in\{g, \phi, \psi\}$  does not possess any nontrivial bounded solution.
{\cb There} exist positive constants $\eps_0=\eps_0(n,\mathcal{F})$ and $C=C(n,\mathcal{F})$ such that, for any 
$\eps\in(0,\eps_0)$ and any strong solution $u$ of
\be{eqnF}
\mathcal{L}u=f_\eps(u)\quad\hbox{in $D$,}
\ee 
we have the estimate
\be{UBfeps}
{f_\eps(u(X))\over u(X)}\le {C(n,\mathcal{F})\over {\rm dist}^2(x,\partial D)},\quad X\in D.
\ee 
In particular, problem \eqref{eqnF} with $D=Q$
does not admit any nontrivial  solution for $\eps\in(0,\eps_0)$.
\end{thm}

\medskip

\begin{rem}
(i) It follows from \eqref{hypgeps}-\eqref{hyppsieps} that, for each $s>0$,
$\lim_{\lambda\to 0}{g(\lambda s)\over g(\lambda)}=\phi(s)$ and $\lim_{\lambda\to\infty}{g(\lambda s)\over g(\lambda)}=\psi(s)$,
i.e.~$g$ has regular variation at $0$ and $\infty$ and
consequently (cf.~after~\eqref{limitg3}),
$\phi,\psi$ must necessarily be power functions $s^m$, with $m>1$ in view of \eqref{hypsuperlineps}.

\smallskip

(ii) The proof shows that the conclusions of Theorem~\ref{thm2eps} remain valid for any strong solution~of 
$$(1-\eta)f_\eps(u)\le\mathcal{L}u\le (1+\eta)f_\eps(u),\quad x\in D,$$
where $\eta>0$ small depends only on $n$ and $\mathcal{F}$.
See also \cite{WY, FIM} and \cite[p.174]{QSb} for related results on stability for the Liouville property.
\end{rem}

{\cb Finally, we have the following uniform version of Theorem~\ref{thm1-0},
assuming regular variation of $f_\eps$ only at infinity.

\begin{thm} \label{thm2epsB}
Assume \eqref{Lell} or \eqref{Lparab}. Let $\mathcal{F}$ satisfy \eqref{hypfeps} and,
for some $m\in(1,p_S)$,
\be{hyppsiepsB}
\lim_{\eps\to 0,\, \lambda\to \infty}{f_\eps(\lambda s)\over f_\eps(\lambda)}=s^m,
\quad\hbox{uniformly for $s$ in compact subsets of $(0,\infty)$.}
\ee
In addition assume
\be{hypfunifbdd}
\hbox{$\ds\sup_{\eps\in(0,\eps_1),\, s\in [0,A]} f_\eps(s)<\infty\ \ $ for each $A>0$,}
\ee
and the superlinearity condition
\be{hypsuperlineps2}
\inf_{\eps\in(0,\eps_1),\,\lambda\ge 1}\ {f_\eps(\lambda s)\over sf_\eps(\lambda )}\to \infty, \quad s \to \infty.
\ee
\smallskip
Then there exists a positive constant $\eps_0=\eps_0(n,\mathcal{F})$ such that, for any 
$\eps\in(0,\eps_0)$, assertion~(i) (resp.,~(ii)) of Theorem~\ref{thm1-0} remains valid with $f$ replaced by $f_\eps$
and $C$ replaced by $C(n,\mathcal{F})$ (resp.,~$C(n,\Omega,\mathcal{F})$).
\end{thm} 
 }

\begin{rem}
 {\bf Known and new elliptic Liouville type theorems.}
We point out that weaker versions of Theorem~D 
(with assumptions on $f$ stronger than \eqref{condLZ})
were obtained prior to \cite{LZ} in \cite{GSa,RZ}.
On the other hand, Liouville type properties for \eqref{eqf} are known to be true under different assumptions
that cannot be compared with \eqref{condLZ}.
Namely, let $p>1$ satisfy $p<p_S$ if $n\le 4$ and $p<(n-1)/(n-3)$ ($<p_S$) if $n\ge 5$.
Assume that 
\be{condNHM1}
2nF(s)-(n-2)sf(s)\ge c_3 s^{p+1}, \quad s>0,\quad\hbox{where } F(s)=\int_0^s f(z)\, dz,
\ee
and that
\be{condNHM2}
c_1s^p\le f(s)\le c_2s^p,\quad s>0,
\ee
for some $c_1,c_2,c_3>0$.
Then, as a consequence of \cite[Theorem~2]{S12}, \eqref{eqf} does not admit any nontrivial bounded solution in $\R^n$
(or even exponentially bounded solutions).

As mentioned before, the proof of Theorem~A (\cite{GSa, BVV}) is based on integral estimates and Bochner's formula,
whereas Theorem~D (cf.~\cite{LZ} and see also \cite{CLi,RZ}) is proved by moving planes or moving spheres methods.
As for the proof in \cite{S12} it is based on Pohozaev's identity, functional inequalities on $S^{n-1}$
and measure arguments.\footnote{Although, as a drawback, it reaches the optimal Sobolev exponent only in dimensions $n\le 4$,
it applies well to certain elliptic systems for which other methods are not always applicable
(see \cite{SZ, S09, S12}).}

On the other hand, Theorem~\ref{thm2eps} (or Theorem~\ref{thm1} under assumption \eqref{hypUB1c})  
produces new Liouville results for strongly oscillating nonlinearities,
that do not follow from known ones.
For instance, it applies for 
$$f(s)=s^{p+a\sin[\log\log(3+s)]}e^{a\sin s}
\quad\hbox{or}\quad
f(s)=s^p\log^a(3+s+\sin(s^2))$$
with $p\in (1,p_S)$ and small $a\in\R$.
Elementary calculations show that $sf'/f$ is unbounded,
whatever the value of $a\ne 0$, so that in particular $s^{-p_S}f(s)$ is not nonincreasing 
and Theorem~D does not apply (and \eqref{condNHM2} does not hold either).
 \end{rem}

\section{Ideas of proofs: Heuristics for generalized scaling arguments} \label{SecHeurist}

The following motivates our assumptions on $f$
and can also serve as a guideline to the proof.
Let us consider the elliptic case for simplicity (the parabolic case is similar).
The principle of Liouville theorems based scaling methods is to look for a rescaled solution of the form
 $$v_k(y)=m_k^{-1}u_k(x_k+\alpha_ky),$$
 where $x_k$ is a suitable point (where the sought-for bound is assumed to be violated for contradiction).
Here $m_k$ and $\alpha_k$ are the rescaling parameters, and the function $v_k$ solves
$$-\Delta v_k(y)=g_k(v_k(y)):=\alpha_k^2m_k^{-1}f(m_k v_k(y)).$$
When $f$ is a power (or asymptotically power-like) function, a natural relation between $m_k$ and $\alpha_k$
is immediately provided by the scale invariance.
However for general $f$ this is no longer the case and $m_k, \alpha_k$ a priori offer two degrees of freedom.
To pass to the limit and end up with a Liouville theorem for a suitable equation, we need that:
 
\vskip 1pt

$\bullet$ $v_k$ stays locally bounded (and does not vanish identically),

\vskip 1pt

$\bullet$ $g_k(s)$ converges to some limit $g(s)>0$ for each $s>0$.
\vskip 1pt

\noindent As usual, the nonvanishing constraint leads to the choice $m_k=u_k(x_k)$ (up to an unimportant multiplicative constant),
and the doubling lemma \cite{PQS1} will then provide the local boundedness.
Next, the second constraint requires $\alpha_k$ to be such that
\be{limitg}
g(s):= \lim_k \, \alpha_k^2m_k^{-1}f(m_ks) \ \hbox{ exists for all $s>0$}.
\ee
Since the case when $m_k$ stays bounded and bounded away from $0$ is not expected to pose significant problems,
we must essentially deal with the cases $m_k\to\infty$ or $0$.
Since \eqref{limitg} must hold in particular with $s=1$, this imposes
\be{limitg2}
\alpha_k\sim\Bigl(\frac{m_k}{f(m_k)}\Bigr)^{1/2}
\ee
(again up to a multiplicative constant). Returning to \eqref{limitg}, it follows that
the limit of $\frac{f(m_ks)}{f(m_k)}$ should exist for all $s>0$ and, since the a priori unknown
sequence $m_k\to 0$ or $\infty$ can be arbitrary, we thus need that
\be{limitg3}
g_\pm(s)=\lim_{\lambda \to \infty \atop ({\rm resp.,}\, \lambda \to0)} \, \frac{f(\lambda s)}{f(\lambda)} >0\ 
\hbox{ exists for all $s>0$}.
\ee
But it is known (see,~e.g.,~\cite{Se}) that functions $f$ with such property are exactly the functions with regular variation
and that the limits must necessarily be power functions $s^q$. 

Now, according to whether $m_k$ goes to $0$ or $\infty$, or remains bounded (along a subsequence),
and after passing to the limit by elliptic estimates, we end up with a nontrivial bounded entire solution $v$
of $-\Delta v=v^q$, or of $-\Delta v=f(v)$,
and we reach a contradiction with the Liouville property either for $f$ or for a power.
Note that in view of the choice \eqref{limitg2} and of the superlinearity of $f$, 
the natural quantity to be estimated turns out to be $f(u)/u$ rather than $u$ in the usual case.

Alternatively, one can avoid assuming
 the existence of the limit in \eqref{limitg} if one can guarantee a precise {\it lower bound} on 
$g_k(v_k)$ (as well as some upper bound to ensure compactness),
so as to end up with a Liouville theorem for supersolutions instead of solutions.
This leads to functions with controled variation and also enjoying the almost decreasing property
 (cf.~\eqref{hypUB1b}).

As for power functions multiplied by a function with controled variation to a small exponent (cf.~\eqref{hypUB1c}),
they arise when combining the above approach with a limiting argument.

\section{Proof of Theorems~\ref{thm1-0}-{\cb\ref{thm1c}} and \ref{thm2eps}{\cb-\ref{thm2epsB}}}  \label{SecProofUB}

\subsection{Preliminaries} \label{SubSecPrelim}

We first recall the doubling lemma from \cite{PQS1}.

\begin{lem} \label{lemDoubling}
Let $({\cb E},d)$ be a complete metric space and let
$\emptyset\ne D\subset\Sigma\subset {\cb E}$, with $\Sigma$ closed. Set 
$\Gamma=\Sigma\setminus D$.
Let $M:D\to(0,\infty)$ be bounded on compact subsets of $D$ 
and fix a real $k>0$.
If  $y\in D$ is such that
$$M(y)\,{\rm dist}(y,\Gamma)>2k,
$$
then there exists $x\in D$ such that
$$
M(x)\,{\rm dist}(x,\Gamma)>2k,\qquad M(x)\geq M(y),
$$
and
$$M(z)\leq 2M(x)\quad\hbox{ for all $z\in D\cap\overline
B_{\cb E}\bigl(x,k\,M^{-1}(x)\bigr)$}.$$
\end{lem}

Next we collect some useful properties of functions with regular or controled variation.
It is known
(see~\cite[Section 1.2]{Se}) that a positive function $L\in C(0,\infty)$ has slow variation at~$\infty$ (resp.,~$0$) iff there exist
continuous functions $\eta,\xi$ such that
 $\ds\lim_{z\to \infty \atop ({\rm resp.,}\,0)} \eta(s)=\lim_{z\to \infty \atop ({\rm resp.,}\,0)} \xi(s)=0$ and
\be{SlowVarRepres}
L(s)=\exp\Bigl[\eta(s)+\int_1^s z^{-1}\xi(z)\,dz\Bigr],\quad s>0. 
\ee
As a consequence of \eqref{SlowVarRepres}, we have 
\be{SlowVarLog}
\log L(s)=o(|\log s|)\quad\hbox{as $s\to\infty$ (resp.,~$0$).}
\ee
Also, $L$ has slow variation whenever it is $C^1$ near $\infty$ (or $0$) and
$\ds\lim_{s\to \infty \atop ({\rm resp.,}\,0)} \ts \frac{sL'(s)}{L(s)}=0$.
\medskip

On the other hand (see~\cite[Appendix]{Se}), $L\in C(0,\infty)$ has controled variation
iff \eqref{SlowVarRepres} holds for some continuous functions $\eta,\xi$ that are bounded  on $(0,\infty)$.
As a consequence, we have 
\be{ControlVarLog}
\log L(s)=O(|\log s|)\quad\hbox{as $s\to\infty$ or $0$.}
\ee
Also, $L$ has controled variation whenever it is $C^1$, positive on $(0,\infty)$
 and $\frac{sL'(s)}{L(s)}$ is bounded. 
 
 \medskip

We shall also use the following technical lemma.

\begin{lem} \label{lemRegVar}
Let $f$ satisfy \eqref{hypcontpos}.
 
 \smallskip

(i) If $f$ has regular variation at $\infty$ with index $m>1$,
then there exist $q>1$ and $c>0$ such that 
$$
\inf_{\lambda\ge 1} \frac{f(s\lambda)}{f(\lambda)}\ge cs^q
\quad\hbox{ for all $s\ge 1$.}
$$

\smallskip

(ii) Assume {\cb that}  
$f$ has regular variation at $0$ and $\infty$ with indices $m_1,m_2>1$.
Then there exist $q>1$ and $c>0$ such that 
\be{lemRegVar3}
\inf_{\lambda>0} \frac{f(s\lambda)}{f(\lambda)}\ge cs^q
\quad\hbox{ for all $s\ge 1$.}
\ee
In particular, $f$ is superlinear (cf.~\eqref{hypsuperlin}).

\smallskip

(iii) Let $L$ be a function with controled variation and $p>1$. 
Then, for $a>0$ sufficiently small, the function $f(s)=s^pL^a(s)$ satisfies
\eqref{lemRegVar3} with some $q>1$ and $c>0$ independent of~$a$.

\smallskip

(iv) If $f$ is superlinear,
then there exist $q>1$ and $C, c>0$ such that 
$$\hbox{$f(s)\le Cs^q$ for all $s\in [0,1]$\quad and\quad
 $f(s)\ge cs^q$ for all $s\ge 1$.}$$
\end{lem}

\begin{proof}
(i) By \eqref{SlowVarRepres} we have
$$f(s)=\tau(s)s^m\exp\Bigl[\int_1^s z^{-1}\xi(z)\,dz\Bigr],\quad s\ge 1,$$
where $C_1\le\tau(s)\le C_2$ for some constants $C_1, C_2>0$
and $\lim_{s\to\infty}\xi(s)=0$.
Fix $0<\eta<m-1$ and pick $s_0>1$ such that $|\xi(s)|\le\eta$ for all $s\ge s_0$. 
For all $\lambda, s>1$, we have
$$
\int_\lambda^{s\lambda} z^{-1}|\xi(z)|\,dz
\le \int_1^{s_0} z^{-1}|\xi(z)|\,dz+\eta\int_{s_0\vee\lambda}^{s(s_0\vee\lambda)} z^{-1}\,dz
=C+\eta\log s$$
hence 
$$\frac{f(s\lambda)}{f(\lambda)}\ge C_1C_2^{-1}{\cb s^m} \exp\Bigl[\int_\lambda^{s\lambda} z^{-1}\xi(z)\,dz\Bigr]
\ge C_1C_2^{-1}e^{-C}s^{m-\eta}.$$

\smallskip
(ii) Denote $\omega_1=(0,1]$, $\omega_2=(1,\infty)$. By \eqref{SlowVarRepres} we have
$$f(s)=\tau_i(s)s^{m_i}\exp\Bigl[\int_1^s z^{-1}\xi_i(z)\,dz\Bigr],\quad s\in \omega_i,\ i\in\{1,2\}$$
where $C_1\le\tau_i(s)\le C_2$ for some constants $C_1, C_2>0$
and $\lim_{s\to 0}\xi_1(s)=\lim_{s\to\infty}\xi_2(s)=0$.
Set $m=\min(m_1,m_2)$ and fix $0<\eta<m-1$. Let $\lambda>0$ and $s>1$.

$\bullet$ First consider the case when $\lambda,s\lambda$ are on the same side of $1$, i.e.~$\lambda,s\lambda\in \omega_i$ for some $i\in\{1,2\}$. Then
$$\frac{f(s\lambda)}{f(\lambda)}
=\frac{\tau_i(s\lambda)}{\tau_i(\lambda)}\frac{(s\lambda)^{m_i}}{\lambda^{m_i}} \exp\Bigl[\int_\lambda^{s\lambda} z^{-1}\xi_i(z)\,dz\Bigr]
\ge \frac{C_1}{C_2}s^{m_i} \exp\Bigl[\int_\lambda^{s\lambda} z^{-1}\xi_i(z)\,dz\Bigr].$$
Pick $0<s_1<1<s_2$ such that $|\xi_1({\cb z})|\le\eta$ for all ${\cb z}\le s_1$
and $|\xi_2({\cb z})|\le\eta$ for all ${\cb z}\ge s_2$. 
Denote $J=\int_{s_1}^{s_2} z^{-1} |\xi_i(z)|\,dz$,  
$\tilde \lambda=\lambda\wedge (s^{-1}s_1)$ if $i=1$, $\tilde \lambda=\lambda\vee s_2$ if $i=2$.
Therefore,
$$
\exp\Bigl[\int_\lambda^{s\lambda} z^{-1}\xi_i(z)\,dz\Bigr]
\ge \exp\Bigl[-J-\int_{\tilde \lambda}^{s\tilde \lambda} z^{-1}|\xi_i(z)|\,dz\Bigr]
\ge C\exp\Bigl[-\eta \int_{\tilde \lambda}^{s\tilde \lambda} z^{-1}\,dz\Bigr]=Cs^{-\eta}
$$
(where $C$ denotes a generic positive constant independent of $\lambda, s$), hence
\be{lemRegVarEq1}
\frac{f(s\lambda)}{f(\lambda)}
\ge C_1C_2^{-1}s^{m_i} Cs^{-\eta}=Cs^{m_i-\eta}.
\ee

$\bullet$ Otherwise we have $\lambda\le 1\le s\lambda$.
Thus $f(s\lambda)\ge C_3(s\lambda)^{m_2-\eta}$ and $f(\lambda)\le C_4\lambda^{m_1-\eta}$
for some constants $C_3, C_4>0$, hence
\be{lemRegVarEq2}
\frac{f(s\lambda)}{f(\lambda)}\ge C_3C_4^{-1} (s\lambda)^{m_2-\eta}\lambda^{-m_1+\eta}
= Cs^{m_2-\eta}\lambda^{m_2-m_1}\ge Cs^{m-\eta}
\ee
(indeed, if $m_2-m_1\le 0$ then $\lambda^{m_2-m_1}\ge 1$, whereas if $m_2-m_1\ge 0$ then $s^{m_2-\eta}\lambda^{m_2-m_1}
\ge s^{m_2-\eta}s^{m_1-m_2}=s^{m_1-\eta}$).

The assertion follows from \eqref{lemRegVarEq1} and \eqref{lemRegVarEq2}.

\smallskip
(iii) We know (cf.~before \eqref{ControlVarLog}) that
$$L(s)=\tau(s)\exp\Bigl[\int_1^s z^{-1}\xi(z)\,dz\Bigr],\quad s>0,$$
where $C_1\le\tau(s)\le C_2$ and $|\xi(s)|\le k$ for some constants $C_1, C_2, k>0$.
Therefore, for $a\in(0,1)$,
$$\frac{f(s\lambda)}{f(\lambda)}
=s^p\Bigl(\frac{L(\lambda s)}{L(\lambda)}\Bigr)^a
\ge [C_1/C_2]^a s^p\exp\Bigl[-\int_\lambda^{s\lambda} ak z^{-1}\,dz\Bigr]
\ge [C_1/C_2] s^{p-ak}.$$
The conclusion follows by taking $a<\frac{p-1}{k}\wedge 1$.

\smallskip

(iv) By assumption, there exists $\sigma>1$ such that
$$\inf_{\lambda>0}\ {f(\lambda \sigma)\over \sigma f(\lambda)}\ge 2.$$
Therefore, setting $q=1+(\ln 2/\ln \sigma)$, we get 
$f(\sigma^{-1}\lambda)\le (2\sigma)^{-1} f(\lambda)=\sigma^{-q}f(\lambda)$ for all $\lambda>0$.
We deduce that, for any nonnegative integer $j$,
$$f(\sigma^{-j}\lambda)\le C(\sigma^{-j})^q\ \hbox{ for all $\lambda\in [\sigma^{-1},1]$}
\quad\hbox{and}\quad
f(\sigma^{j}\lambda)\ge c(\sigma^j)^q\ \hbox{ for all $\lambda\in [1,\sigma]$}.$$
Consequently,
$f(s)\le C\sigma^qs^q$ for all $s\in[\sigma^{-j-1}, \sigma^{-j}]$
and
$f(s)\ge c\sigma^{-q}s^q$ for all $s\in[\sigma^j, \sigma^{j+1}]$.
The conclusion follows.
\end{proof}

For the proof of Theorem~\ref{thm1} under assumption \eqref{hypUB1b}, 
we shall also use the following well-known nonexistence results for inequalities.

\begin{lem} \label{lemGidFuj}
Let $c>0$.
\smallskip

(i) For $p\in (1,p_{sg}]$, the inequality $-\Delta u\ge cu^p$ in $\Rn$ admits no nontrivial (strong) solution.

\smallskip

(ii) For $p\in (1,p_F]$, the inequality 
\be{eqGidFuj1}
u_t-\Delta u\ge cu^p\quad\hbox{in $\Rn\times(0,\infty)$}
\ee
admits no nontrivial (strong) solution.
\end{lem}

See, e.g.,~\cite[Sections~2 and 26]{MP}; the results remain valid for suitable weak solutions.
On the other hand, we note that the range of $p$ cannot be enlarged for
solutions of \eqref{eqGidFuj1} in $\Rn\times(-\infty,\infty)$ instead of $\Rn\times(0,\infty)$,
as shown by the counter-examples in \cite{Tal09, Kur} and \cite[Proposition~21.14]{QSb}.

\medskip

\subsection{Proof of Theorem~\ref{thm1-0}{\cb\rm(i)}, of Theorem~\ref{thm1} under assumption \eqref{hypUB1b} and of Theorem~\ref{thm1c}.}  
\label{ProofThm1}
First note that, under assumption \eqref{hypUB1b} of Theorem~\ref{thm1} or \eqref{HypLiouvilleReg} of Theorem~\ref{thm1c},
Lemma~\ref{lemRegVar}(ii){\cb (iv)} 
guarantees that
\be{DoublingA00}
\hbox{$g(s):=\ds\frac{f(s)}{s}$ can be extended to a continuous function on $[0,\infty)$}
\ee
with $g(0)=0$.

\begin{proof}[Proof in the parabolic case]
We set 
\be{DefMcases} 
M(u)=\Bigl[\ds{f(u)\over {\cb\sigma}+u}\Bigr]^{1/2},\quad\hbox{\cb with } \sigma=
\begin{cases}
1 \quad&\hbox{under the assumptions of Theorem~\ref{thm1-0}{\cb\rm(i)},} \\
\noalign{\vskip 1mm}
0 \quad&\hbox{otherwise}.
\end{cases}
\ee
The sought-for estimates \eqref{EstimA-0} and \eqref{EstimA} write equivalently as
$$M(u(X))\le C\bigl(\sigma+d^{-1}(X,\partial D)\bigr),\quad X\in D$$
with $C=C(n,f)>0$.
Assume for contradiction that there exist sequences of
 domains $\Omega_k$, times $T_k>0$, solutions $u_k$ of 
\be{ContradA1} 
\mathcal{L}u_k=f(u_k),\quad X\in D_k:=\Omega_k\times(0,T_k),
\ee
and points $Y_k\in D_k$ such that
\be{ContradA2} 
M_k(Y_k)>2k\bigl(\sigma+d^{-1}(Y_k,\partial D_k)\bigr)\ge 2kd^{-1}(Y_k,\partial D_k)
\ee
where $M_k=M(u_k)$. We note that that the function $M_k$ is continuous
(owing to \eqref{DoublingA00} in the second case of \eqref{DefMcases}), hence locally bounded on $D_k$.

We use Lemma~\ref{lemDoubling} with 
${\cb E=\,}\R^{n+1}$, equipped with the parabolic distance $d=d_P$, 
$\Sigma=\Sigma_k=\overline D_k$, $D=D_k$, and  $\Gamma=\partial D_k$.
Consequently, there exist points $X_k\in D_k$ such that
\be{DoublingA0}
M_k(X_k)>2kd^{-1}(X_k,\partial D_k),\quad M_k(X_k)\ge M_k(Y_k)>2k\sigma
\ee
and 
\be{DoublingA}
M_k(X)\le 2M_k(X_k)\quad\hbox{ in $\hat D_k:=\bigl\{X;\, d(X,X_k)\le kM_k^{-1}(X_k)\bigr\}$}
\ee
(note that $\hat D_k\subset D_k$ owing to the first part of \eqref{DoublingA0}).
{\cb Set
$$m_k=u_k(X_k).$$
We also} have
\be{DoublingAspecial}
m_k\to\infty\quad\hbox{under the assumptions of Theorem~\ref{thm1-0}{\cb\rm(i)}}
\ee
(due to the continuity of $f$ and the fact that $f(m_k)\ge 4k^2$ by \eqref{DoublingA0})
{\cb and we may thus assume}
\be{DoublingAspecial2}
{\cb m_k\ge 1\quad\hbox{under the assumptions of Theorem~\ref{thm1-0}{\rm(i).}}}
\ee

\smallskip

We now define the rescaled solutions. Setting 
$$  
\alpha_k= M_k^{-1}(X_k),$$
we let
$$v_k(Y)=m_k^{-1}u_k(T_k(Y))\quad\hbox{\cb for $Y\in \tilde D_k$,}$$
where 
$$\tilde D_k=\bigl\{Y=(y,\tau)\in\R^n\times\R:|y|<k/2,\ |\tau|<k^2/4\bigr\}$$
and
$$T_k(Y)=(x_k+\alpha_k y,t_k+\alpha^2_k \tau),\quad\hbox{\cb with $X_k=(x_k,t_k)$}.$$
We claim that there exists $A>0$ (independent of $k$) such that 
\be{DoublingA1}
v_k(Y)\le A\quad\hbox{ in $\tilde D_k$.}
\ee
To see this, for $Y\in \tilde D_k$, set 
\be{DoublingNotation}
X=T_k(Y)\in \hat D_k,\quad \lambda=m_k,\quad s=v_k(Y).
\ee
In the second case of \eqref{DefMcases}, for any $Y$ such that $v_k(Y)>0$, we have 
\be{DoublingA2}
{f(\lambda s)\over sf(\lambda)}={f(\lambda s)\over \lambda s}{\lambda\over f(\lambda)}
={f(u_k(X))\over u_k(X)}{m_k\over f(m_k)}={M_k^2(X)\over M_k^2(X_k)}\le 4
\ee
by \eqref{DoublingA}. Therefore \eqref{DoublingA1} follows from 
{\cb Lemma~\ref{lemRegVar}(i) and \eqref{DoublingAspecial2}
under the assumptions of Theorem~\ref{thm1-0}(i),
from Lemma~\ref{lemRegVar}(ii)
under the assumptions of Theorem~\ref{thm1c},
or else by assumption~\eqref{hypUB1b}.}
In the first case of \eqref{DefMcases}, for any $Y$ such that $v_k(Y)>0$, we have
$$
{f(\lambda s)\over sf(\lambda)}={f(u_k(X))\over u_k(X)}{m_k\over f(m_k)}
\le 
{f(u_k(X))\over 1+u_k(X)}{1+m_k\over f(m_k)}\Bigl(1+{m_k\over u_k(X)}\Bigr)
=\Bigl(1+{m_k\over u_k(X)}\Bigr){M_k^2(X)\over M_k^2(X_k)},
$$
hence
\be{DoublingA2b}
{f(\lambda s)\over sf(\lambda)}\le 4\bigl(1+s^{-1}\bigr)
\ee
and the existence of $A$ then follows from Lemma~\ref{lemRegVar}(i).

Keeping the notation \eqref{DoublingNotation}, we then compute
\be{DoublingA3a}
\mathcal{L}v_k(Y)
=\alpha_k^2m_k^{-1}f\bigl(u_k(T_k(Y))\bigr)
={\cb(1+\sigma m_k^{-1})}{f\bigl(u_k(T_k(Y))\bigr)\over f(m_k)}={\cb(1+\sigma m_k^{-1})}{f(\lambda s)\over f(\lambda)}.
\ee
{\cb Consequently, using \eqref{DoublingAspecial2}}--\eqref{DoublingA2b}, we obtain
\be{DoublingA3}
0\le \mathcal{L}v_k(Y)\le {\cb 8}(A+1), \quad Y\in \tilde D_k.
\ee
Moreover, we have $v_k(0,0)=1$.
By \eqref{DoublingA1}, \eqref{DoublingA3} and parabolic estimates, it follows that, up to extracting a subsequence, 
$v_k$ converges locally uniformly to a function 
$w\in W^{2,1;r}_{loc}(\Rn\times\R)$, $1<r<\infty$, which satisfies
$\mathcal{L}w\ge 0$ and $0\le w\le A$ in $Q$, along with $w(0,0)=1$. 
The strong maximum principle yields that, for some $\tau_0\in[-\infty,0)$, we have 
\be{DoublingA3b}
\hbox{$w=0$ in $\Rn\times(-\infty,\tau_0]$ and $w>0$ in $\Rn\times(\tau_0,\infty)$.}
\ee

We then argue separately according to our different assumptions:
\medskip

$\bullet$ {\it Case of Theorem~\ref{thm1} under assumption \eqref{hypUB1b}.}
Our assumption guarantees that
$$\theta_A:=\inf_{\lambda>0\atop 0<s\le A} \frac{f(\lambda s)}{s^p f(\lambda)}>0. 
$$
Applying this with $\lambda=m_k$ and $s=v_k(Y)\le A$ (hence $u_k(T_k(Y))=\lambda s$) guarantees
$$
{f\bigl(u_k(T_k(Y))\bigr)\over v_k^{p}(y)f(m_k)} \ge c:=\theta_A>0
$$
hence
$\mathcal{L}v_k(Y)\ge cv_k^{p}(Y)$, for all $Y\in \tilde D_k$, by \eqref{DoublingA3a}.
It follows that $w$ is a nontrivial bounded (strong) solution of
\be{DoublingLiouv1}
\mathcal{L}w\ge cw^p \quad\hbox{in $Q$,}
\ee
which contradicts Lemma~\ref{lemGidFuj}(ii).

\smallskip

$\bullet$ {\it Case of Theorems~\ref{thm1-0}{\cb\rm(i)} and \ref{thm1c}.}
Let $q_1$ (resp.~$q_2$) $\in (1,p_S)$ be the index of regular variation of $f$ at $\infty$
(resp., at $0$, under the assumptions of Theorem~\ref{thm1c}).
At least one the following cases must occur, for some subsequence:

\medskip

(i) $m_k\to \infty$;

\smallskip

(ii) $m_k\to 0$;

\smallskip

(iii) $m_k\to a\in (0,\infty)$.

\medskip

First consider case (i). 
It is known (see \cite[Theorem~1.1]{Se})
that the convergence in \eqref{hypregulvar} is uniform for $s$ in compact sets of $(0,\infty)$.
It follows from \eqref{DoublingA3b} that
\be{DoublingLiouv2a}
\lim_k{f(m_k v_k(Y))\over f(m_k)}=w^q(Y),
\quad\hbox{for each $Y\in \Rn\times(\tau_0,\infty)$,}
\ee
with $q=q_1$. {\cb By \eqref{DoublingA3a} (and using \eqref{DoublingAspecial} in case $\sigma=1$), we deduce}
 that $w$ is a bounded (strong hence) classical solution of
\be{DoublingA3b2}
\mathcal{L}w=w^q \quad\hbox{in $\Rn\times(\tau_0,\infty)$.}
\ee
Moreover, we cannot have $\tau_0>-\infty$, since otherwise 
$w(\cdot,\tau_0)\equiv 0$ by \eqref{DoublingA3b} and $\mathcal{L}w\le Cw$, hence $w\equiv 0$ by the maximum principle:
a contradiction.
Consequently, $\tau_0=-\infty$, which is a contradiction with Theorem~B.
Note that {\cb if $\sigma=1$, i.e.,} under the assumptions of Theorem~\ref{thm1-0}{\cb\rm(i)},
only case~(i) occurs, {\cb due to \eqref{DoublingAspecial}.}
Therefore, Theorem~\ref{thm1-0}{\cb\rm(i)} in the parabolic case is proved.
In case~(ii), exactly the same argument applies with $q_2$ instead of $q_1$,
hence again a contradiction.

Finally consider case (iii). For each $Y\in Q$, owing to the continuity of $f$, we have
$$\lim_k{f(m_k v_k(Y))\over f(m_k)}={f(aw(Y))\over f(a)}.$$
Therefore, setting $\mu=\sqrt{f(a)/a}$ and replacing $w$
with $\tilde w(y,s)=aw(\mu y,\mu^2 s)$, 
we get a bounded strong solution of 
\be{DoublingLiouv3}
\mathcal{L}w=f(w) \quad\hbox{in $Q$,}
\ee
which contradicts assumption~\eqref{HypLiouville}.
Theorem~\ref{thm1c} in the parabolic case is proved.
\end{proof}

\begin{rem} \label{remellpar}
We observe that, in the above proof, if the solutions $u_k$ are assumed to be time independent, then so is the limiting function $w$.
Therefore, for a given $f$, the universal estimate for time independent solutions holds whenever
the corresponding elliptic Liouville properties are available. 
\end{rem}

\begin{proof}[Proof in the elliptic case]
Note that any solution can be regarded as a solution of the corresponding parabolic problem at time $t=1$ with, say, $T=2$. 
In view of Remark~\ref{remellpar}, the desired conclusions follow
from Lemma~\ref{lemGidFuj}(i), or Theorem~A and/or assumption \eqref{HypLiouville}.
\end{proof}

\subsection{Proof of {\cb Theorem~{\ref{thm1-0}{\rm(ii)}}}, of Theorem~\ref{thm1} under assumptions {\cb\eqref{hypUB1c}}, \eqref{hypUB1a}
and of Theorems~\ref{thm1b}, \ref{thm2eps}, {\cb\ref{thm2epsB}}.}

\begin{proof}[Proof of {\cb Theorem~{\ref{thm1-0}{\rm(ii)}}} and of Theorem~\ref{thm1b} under assumption \eqref{hypUB1b}]

(i) {\it Parabolic case.}
The result easily follows by modiying the proof in subsection~\ref{ProofThm1} 
{\cb with $\sigma=1$} along the lines of \cite[Theorem~4.1]{PQS2}.
In addition to the Liouville properties in the whole space, used in subsection~\ref{ProofThm1},
the modified arguments require the following 
Liouville properties in the half-space
$H=\Rn_+\times\R$ (where $\Rn_+=\{(x_1,\dots,x_n)\in \Rn;\, x_1>0\}$).
{\cb Namely, for the case of Theorem~{\ref{thm1-0}{\rm(ii)}}, we use the fact that} the problem
\be{halfspace1b}
\begin{cases}
\mathcal{L}u = u^q \quad&\hbox{in $H$}, \\
\noalign{\vskip 1mm}
u=0 \quad&\hbox{on $\partial H$} 
\end{cases}
\ee
with $q\in (1,p_S)$ has no nontrivial bounded solution
by \cite[Theorem~2.1(i)]{PQS2}\footnote{The boundedness assumption can actually be removed, see~\cite{Q21b}.}.
{\cb For the} case of Theorem~\ref{thm1b} under assumption \eqref{hypUB1b}, {\cb we use the fact that,} 
since $p\le p_{**}=1+2/(n+1)$, 
the inequality
\be{halfspace2}
\mathcal{L}u \ge u^p \quad\hbox{in $H$}
\ee
has no nontrivial strong solution (see \cite[Section~32]{MP}).

\smallskip

(ii) {\it Elliptic case.}
As before, any solution can be regarded as a solution of the parabolic problem at time $t=1$ with $T=2$. 
Also, by \eqref{hypsuperlin} with $\lambda=1$
(which is true either by assumption \eqref{hypUB1b} or by Lemma~\ref{lemRegVar}),
it is sufficient to prove that $f(u)/u\le C$ instead of $u\le C$ in \eqref{EstimB}.
In view of Remark~\ref{remellpar}, 
the desired conclusion follows provided the elliptic Liouville properties in the whole space and in the half-space are satisfied.

For the whole space, they are satisfied, {\cb respectively} owing to {\cb Theorem~A and}
Lemma~\ref{lemGidFuj}(i) (noting that $p_{**}\le p_*$). For the half space, the elliptic Liouville properties corresponding to
{\cb\eqref{halfspace1b}}-\eqref{halfspace2} are also true, respectively owing 
 to {\cb \cite[Theorem~1.1]{GSb} and \cite{BCDN}.}
\end{proof}

\begin{rem} \label{remGSnodoubling}
In the elliptic case, if we consider only the case $\Omega$ bounded, then 
our modified rescaling procedure (based on $f(u)/u$) can be applied without using the doubling Lemma
since, as in \cite{GSb}, it then suffices to work with points $x_k$ where $u_k$ achieves its maximum.
However this argument is not sufficient for unbounded domains
 since a given solution to \eqref{eqf} with $u=0$ on $\Gamma$ might then a priori be unbounded.
\end{rem}

\begin{proof}[Proof of {\cb Theorem~\ref{thm1} } 
under assumption \eqref{hypUB1a}] 
 {\cb This is a direct consequence of Theorem~\ref{thm1c} 
provided property \eqref{HypLiouville} is satisfied.} In the elliptic case, this is true for $n\ge 3$ by Theorem~D,
whereas for $n\le 2$ it follows from the nonexistence of nonconstant nonnegative entire superharmonic functions.
In the parabolic case, {\cb property \eqref{HypLiouville} is guaranteed by}
Theorem~{\cb\ref{S-thmLiouvJLV2}}
(the latter will be independently proved in the next section).
\end{proof}

\begin{proof}[Proof of Theorem~\ref{thm2eps}] 
{\cb It} follows {\cb from a simple} modification of the proof of Theorem~\ref{thm1c}.
Namely, assuming for contradiction that estimate \eqref{UBfeps} fails for some sequence $\eps_k\to 0$,
there exist domains $\Omega_k$, times $T_k>0$, points $Y_k\in D_k$ and solutions $u_k$ of \eqref{ContradA1} 
with $f=f_k:=f_{\eps_k}$ such that \eqref{ContradA2} holds {\cb (with $\sigma=0$)}.
Properties \eqref{DoublingA00} {\cb and \eqref{DoublingA1}} remain valid owing to 
 {\cb assumption~\eqref{hypsuperlineps} (using also Lemma~\ref{lemRegVar}(iv) for \eqref{DoublingA00}).}
The rest of the proof is 
 then unchanged, replacing $f$ by $f_k$, 
and making use of assumptions
\eqref{hyppsieps}, \eqref{hypphieps}, \eqref{hypgeps} to treat cases (i), (ii) and (iii), respectively
{\cred (those cases refer to lines 5-7 after eqn.~\eqref{DoublingLiouv1}).}
\end{proof}

{\cb
\begin{proof}[Proof of Theorem~\ref{thm2epsB}] 
It follows from a simple modification of the proof of Theorem~\ref{thm1-0}.
In particular we note that assumption \eqref{hypfunifbdd} guarantees that $m_k\to\infty$.
Then assumption \eqref{hypsuperlineps2} implies \eqref{DoublingA1} and \eqref{DoublingLiouv2a} 
follows from \eqref{hyppsiepsB}.
\end{proof}
}

\begin{proof}[Proof of Theorems~\ref{thm1} and \ref{thm1b} under assumption \eqref{hypUB1c}] 
{\cb Theorem~\ref{thm1} under assumption \eqref{hypUB1c}} is a special case of Theorem~\ref{thm2eps}. 
Indeed, {\cb since $L$ has controled variation, one easily checks that,
for any compact $\mathcal{K}\subset(0,\infty)$, there exist constants $C_1>C_2>0$ such that
$C_1\le {L(\lambda s)\over L(\lambda)}\le C_2$ for all $s\in \mathcal{K}$ and all $\lambda>0$.}
Then taking $f_\eps(s)=s^p L^\eps(s)$, 
we see that \eqref{hypgeps}-\eqref{hyppsieps} are satisfied with $g(s)=\phi(s)=\psi(s)=s^p$, and that
\eqref{hypsuperlineps} holds for $\eps_{\cb 1}>0$ small owing to Lemma~\ref{lemRegVar}(iii).
Also the Liouville type assumption of Theorem~\ref{thm2eps} holds owing to Theorems~A and~{\cred B}. 

{\cb As for Theorem~\ref{thm1b} under assumption \eqref{hypUB1c}, it is likewise a special case of Theorem~\ref{thm2epsB}
(noting in particular that $f_\eps(s)=s^p L^\eps(s)$ satisfies  \eqref{hypfunifbdd} owing to \eqref{ControlVarLog}).}
\end{proof}

\section{Proof of Theorems~\ref{S-thmLiouvJLV2} and \ref{S-thmLiouvJLV1}} \label{SecProofLiouvParab}

{\cred Throughout this section we shall denote 
$F(s)=\int_0^s f(z)dz$ and $\tilde F(s)=\int_0^s \tilde f(z)dz$ for $s\ge 0$.}
The proof of Theorem~\ref{S-thmLiouvJLV1} uses the following key 
gradient estimate.

\begin{lem} \label{S-lemLiouvBV3}
Assume \eqref{hypcontpos}, \eqref{S-LiouvParabHypInt} and \eqref{S-LiouvParabHypBV}.
Let $\Omega$ be an arbitrary domain in $\Rn$, $T>0$, and
$0\leq \varphi\in\mathcal{D}(\Omega\times (-T,T))$.
Let $u\in W^{2,1;r}_{loc}(\Omega\times (-T,T))$ for all finite $r$, be a 
strong solution of
$u_t-\Delta u=f(u)$ in $\Omega\times (-T,T)$.
Fix $\eps>0$  
and set 
$$v= u+\eps,\quad g_\eps=f(u)-f(u+\eps).$$
Denote 
$$I=\int\int \varphi\,v^{-2}|\nabla v|^4,\quad
L=\int\int \varphi\,f(v)\tilde f(v),
$$
where, here and below, double integrals are over $\Omega\times (-T,T)$.
Then there exists $C=C(n,f)>0$ such that
\be{S-LiouvBVestimIK}
\begin{aligned}
I+ L
\leq\ &C\int\int \varphi\,\bigl[(v_t)^2 +|v_t|\,v^{-1}|\nabla 
v|^2\bigr]+|\nabla v|^2 |\Delta\varphi|\\
&\quad +C\int\int \bigl(\tilde f(v)+|v_t|+v^{-1} |\nabla v|^2+|g_\eps| \bigr)
|\nabla v\cdot\nabla\varphi| + \tilde F(v)|\varphi_t|{\cred \ +F(v)|\Delta\varphi|}\\
&\quad +C\int\int \varphi\,\bigl(v^{-1} |\nabla v|^2+\tilde f(v)+|g_\eps| 
\bigr)|g_\eps| \\
&\equiv T_1+T_2+T_3.
\end{aligned}
\ee
\end{lem}

\begin{rem} \label{remproofeps} 
The use of $u+\eps$ in Lemma~\ref{S-lemLiouvBV3}, instead of $u$, 
will generate some technical complication
(to handle the additional perturbation terms $g_\eps$).
However, in order to rule out the existence of nontrivial solutions of $u_t-\Delta u=f(u)$
in $Q=\Rn\times\R$,
and not only of positive solutions,
this procedure seems necessary here, 
unlike in the special case $f(s)=s^p$.

Recall (cf.~\cite[Remarks 21.3 and 26.10(i)]{QSb}) that in the case $f(s)=s^p$: (a) nonexistence of positive solutions
 in $Q$ yields 
nonexistence of nontrivial {\it bounded} solutions, since such solutions must be positive 
as a consequence of the maximum principle (see \eqref{DoublingA3b} and after \eqref{DoublingA3b2});
(b) then, by rescaling-doubling arguments, this guarantees universal estimates of the form \eqref{EstimA},
and the latter imply the nonexistence of nontrivial solutions in $Q$.

Such a procedure does not apply in general under the assumptions 
of Theorems~\ref{S-thmLiouvJLV2} and~\ref{S-thmLiouvJLV1},
since step (a) requires $f(s)/s$ to be bounded near $s=0$ and step (b) requires
additional assumptions to derive universal estimates (cf.~Theorem~\ref{thm1}). \end{rem}

In view of the proof of Lemma~\ref{S-lemLiouvBV3}, we recall the following lemma from 
\cite{BVV} (see also \cite[Lemma~8.9]{QSb}),
which provides a family of integral
estimates relating any function with its gradient and its Laplacian
(the result is given there for $C^2$ functions but it immediately extends by density).

\begin{lem} \label{S-lemLiouvGS1}
Let $\Omega$ be an arbitrary domain in $\Rn$,
$0\leq \varphi\in \mathcal{D}(\Omega)$, and let
$v\in W^{2,r}_{loc}(\Omega)$ for all finite $r$, with $v>0$.
Fix $q\in \R$ and denote
\be{S-LiouvGSestIJK0}
I_q=\int \varphi\,v^{q-2}|\nabla v|^4,\quad
J_q=\int \varphi\,v^{q-1}|\nabla v|^2 \Delta v,\quad
K_q=\int \varphi\,v^q(\Delta v)^2,
\ee
{\cred where $\int=\int_\Omega$.}
Then, for any $k\in \R$ with $k\neq -1$, there holds
\be{S-LiouvGSestIJK}
\alpha I_q+\beta J_q+\gamma K_q\leq
{1\over 2}  \int \,v^q|\nabla v|^2\Delta \varphi
+ \int v^q\bigl[\Delta v+(q-k)v^{-1}
|\nabla v|^2\bigr]\nabla v\cdot\nabla\varphi,
\ee  
where
$$\alpha=-{n-1\over n}k^2+(q-1)k-{q(q-1)\over 2}, \quad
\beta={n+2\over n}k-{3q\over 2}, \quad
\gamma=-{n-1\over n}.
$$
\end{lem}

\begin{proof}[Proof of Lemma~\ref{S-lemLiouvBV3}]
We apply Lemma~\ref{S-lemLiouvGS1} with $q=0$
 to $v=v(\cdot,t)$ and $\varphi=\varphi(\cdot,t)$ for a.e.~$t$, 
and integrate the corresponding inequality \eqref{S-LiouvGSestIJK} in time. 
Denoting $J={\cred\int_{-T}^T}J_0$, $K={\cred\int_{-T}^T}K_0$, this gives us
\be{S-LiouvBVestIJK}
\begin{aligned}
-\Bigl({n-1\over n}k&+1\Bigr)k I+{n+2\over n}k J -{n-1\over n}K  \\
&\le  {1\over 2}  \int\int \,|\nabla v|^2\Delta \varphi
+ \int\int \bigl[\Delta v-kv^{-1} |\nabla v|^2\bigr]\nabla v\cdot\nabla\varphi.
\end{aligned}
\ee  
Set
$$L^*=\int\int \varphi\,f^2(v).$$
Using 
\be{S-LiouvBVestIJK0}
-\Delta v=f(v)+g_\eps-v_t
\ee
 and integrating by parts
in $t$ and/or in $x$, we obtain
$$\begin{aligned}
K
&=\int\int \varphi\, (v_t)^{2}+\int\int \varphi\,(f(v)+g_\eps)^2-2\int\int 
\varphi\,(f(v)+g_\eps)v_t \\
&=\int\int \varphi\, (v_t)^{2}+L^*+2\int\int F(v)\varphi_t+\int\int \varphi\,(2f(v)+g_\eps-2v_t)g_\eps
\end{aligned}$$
and
$$\begin{aligned}
J
&=-\int\int\varphi\,|\nabla v|^2{f(v)\over v}
+\int\int \varphi\,(v_t-g_\eps)\,{|\nabla v|^2\over v} \\
&=-\int\int \varphi\,\nabla v\cdot\nabla (\tilde f(v))
+\int\int \varphi\,(v_t-g_\eps)\,{|\nabla v|^2\over v} \\ 
&=\int\int \varphi\,(\Delta v)\tilde f(v)
+\int\int (\nabla\varphi\cdot\nabla v)\tilde f(v) +\int\int 
\varphi\,(v_t-g_\eps)\,{|\nabla v|^2\over v} \\
&=-{\cred L} 
-\int\int \tilde F(v)\varphi_t 
+\int\int (\nabla\varphi\cdot\nabla v)\tilde f(v)
+\int\int \varphi\,(v_t-g_\eps)\,{|\nabla v|^2\over v}-\int\int \varphi\,g_\eps\tilde f(v).
\end{aligned}$$
{\cred 
Substituting the above expressions of $K, J$ in \eqref{S-LiouvBVestIJK}
and also replacing $\Delta v$  in \eqref{S-LiouvBVestIJK} via}~\eqref{S-LiouvBVestIJK0}, we obtain
$$\begin{aligned}
&-\Bigl({n-1\over n}k+1\Bigr)k I-{n+2\over n}kL -{n-1\over n}L^* \\
&\leq {1\over 2}  \int\int \,|\nabla v|^2\Delta \varphi
{\cred \ - \int\int f(v)\nabla v\cdot\nabla\varphi}
+ \int\int \bigl[{\cred v_t-g_\eps}-kv^{-1} |\nabla v|^2\bigr]\nabla v\cdot\nabla\varphi \\
&\ \ +{n+2\over n}k\Bigl\{\int\int \tilde F(v)\varphi_t
-\int\int (\nabla\varphi\cdot\nabla v)\tilde f(v)
-\int\int \varphi\,(v_t-g_\eps)\,{|\nabla v|^2\over v}+\int\int \varphi\,g_\eps\tilde f(v)\Bigr\}\\
&\ \ +{n-1\over n}\Bigl\{\int\int \varphi\, (v_t)^{2}+2\int\int F(v)\varphi_t
+\int\int \varphi\,(2f(v)+g_\eps-2v_t)g_\eps\Bigr\}.
\end{aligned}$$
{\cred Next integrating by parts in space the second term of the RHS, and}
using $f\le p\tilde f$ and $F\le p\tilde F$ {\cred in case $n\ge 2$, we obtain}
$$\begin{aligned}
&-\Bigl({n-1\over n}k+1\Bigr)k I-{n+2\over n}kL -{n-1\over n}L^* \\
&\leq  \int\int \ \Bigl(\frac12|\nabla v|^2{\cred\ +F(v)\Bigr)}\Delta \varphi
+ \bigl[|g_\eps|+|v_t|+|k|v^{-1} |\nabla v|^2\bigr] |\nabla v\cdot\nabla\varphi| \\
&\ \ +C(n,k)\int\int \tilde F(v)|\varphi_t|
+ |\nabla\varphi\cdot\nabla v|\tilde f(v)
+\varphi\,(|v_t|+|g_\eps|)\,{|\nabla v|^2\over v}+\varphi\,|g_\eps|\tilde f(v) \\ 
&\ \ +C(n) \int\int \varphi\, [(v_t)^{2} 
+{\cred g_\eps^2}].
\end{aligned}$$
If $n=1$, the choice $k={\cred -2}$ 
 yields \eqref{S-LiouvBVestimIK}.
If $n\ge 2$, by assumption \eqref{S-LiouvParabHypBV}, we have
$$-{n+2\over n}kL -{n-1\over n}L^*\ge 
{n+2\over n}\Bigl(-k -{(n-1)p\over n+2}\Bigr)L{\cb .}$$
Since $p<n(n+2)/(n-1)^2$, we may choose $k<0$, {\cred $k\ne -1$,} such that
$${(n-1)p\over n+2}<-k<{n\over n-1}$$
and \eqref{S-LiouvBVestimIK} follows.
\end{proof}

We shall need the following technical lemma, 
which provides some useful growth properties of $f$.

\begin{lem} \label{lemgrowth}
Let $i\in\{1,2\}$. Under the assumptions of Theorem~\ref{S-thmLiouvJLV1}, 
there exists $c>0$ such that
\be{growthf1}
f(s)\tilde f(s)\ge c s^{2\gamma_i},\ s\in\omega_i,\ \hbox{ for some 
$1<\gamma_2<\gamma_1<p_S$, with $n\ge 2$ if $i=1$,}
\ee
and
\be{growthf2}
f(s) \tilde f(s)\ge c {\cred (F(s)+\tilde F(s))}^{d_i},\ s\in\omega_i,\ \hbox{ for some $1<d_2<d_1<(n+2)/n$.}  
\ee

\end{lem}

\begin{proof}
In the proof $c$ will denote a generic constant depending only on $f$.
First note that, by assumption,
\be{growthf0}
{2\over 3}<m_2<m_1<m^*<d^*:={n+2\over n},
\ee
and $m^*\le 1$ if $n\ge 2$. Also, in view of \eqref{Qbddunbddnonc} and
since $\tilde f$ is increasing, 
we may assume $m_2<1$ also when $n=1$, without loss of generality. 

To prove \eqref{growthf1}, we let $i\in\{1,2\}$, with $n\ge 2$ if $i=1$. By \eqref{Qbddunbddnonc}, we have $\tilde f'(s)=s^{-1}f(s)\ge cs^{m_i-1}\tilde f^{2m_i-1}$ i.e.,
$[\tilde f^{2(1-m_i)}]'(s)\ge cs^{m_i-1}$ for $s\in \omega_i$.
Integrating (and using $\tilde f(s)\ge \tilde f(1)>0$ for $s\ge 1$ in case $i=2$), we get
\be{growthf4}
\tilde f(s)\ge cs^{\frac{m_i}{2(1-m_i)}},\quad s\in\omega_i.
\ee
Applying \eqref{Qbddunbddnonc} again, we obtain
$$f(s)\tilde f(s)\ge cs^{m_i}\tilde f^{2m_i}(s)\ge 
cs^{m_i(1+\frac{m_i}{1-m_i})}=cs^{\frac{m_i}{1-m_i}},\quad s\in\omega_i.$$
By \eqref{growthf0}, we have $\frac{m_1}{2(1-m_1)}<p_S$ and $\frac{m_2}{2(1-m_2)}>1$.
Choosing $\gamma_1\in (\max\{1,\frac{m_1}{2(1-m_1)}\},p_S)$ and then
$\gamma_2\in (1,\min\{\gamma_1,\frac{m_2}{2(1-m_2)}\})$ we obtain \eqref{growthf1}.

Next let $i\in\{1,2\}$, $n\ge 1$. Owing to \eqref{growthf0}, we may choose 
$d_1\in (\max\{1,m_1\},d^*)$ close to $d^*$ 
and then $d_2\in(1,\min\{d_1,4/3\})$ close to $1$ such that
$$\frac{2m_1-d_1}{d_1-m_1}<\frac{2(1-m_1)}{m_1}\quad\hbox{and}\quad
\frac{m_2}{2(1-m_2)}>\frac{d_2-m_2}{2m_2-d_2}.$$ 
It follows from \eqref{growthf4} (using the boundedness of $\tilde f$ on $(0,1)$) that
$$\tilde f^{\frac{2m_1-d_1}{d_1-m_1}}(s) \ge 
{\cred c} \tilde f^{\frac{2(1-m_1)}{m_1}}(s) \ge cs,\ \ 0<s<1
\quad\hbox{ and }\quad
\tilde f(s)\ge cs^{\frac{m_2}{2(1-m_2)}}\ge cs^{\frac{d_2-m_2}{2m_2-d_2}},\ \ s\ge 1.$$
Using \eqref{Qbddunbddnonc}, we deduce 
$$f(s)\tilde f(s) \ge cs^{m_i}\tilde f^{2m_i}(s)\ge c(s\tilde f(s))^{d_i}, \quad s\in\omega_i.$$ 
Since ${\cred \tilde F(s)+F(s)=\int_0^s (\tilde f(z)+z\tilde f'(z))dz=\int_0^s (z\tilde f)'dz=s\tilde f(s),}$
property \eqref{growthf2} follows.
\end{proof}

\begin{proof}[Proof of Theorem~\ref{S-thmLiouvJLV1}]
We split the proof into several steps.
\medskip

{\bf Step 1.} {\it Notation and choice of test-functions.}
Let $u$ be a strong solution of \eqref{S-LiouvBVeqnf}.
Fix $\eps\in(0,1)$, $R>1$ and set $v=u+\eps$.
We also fix $\eta\in(0,1)$. 
In what follows $C_\eta$ denotes a generic positive constant depending on $\eta$
but independent of~$\eps, R$.

We shall estimate the
terms on the RHS of \eqref{S-LiouvBVestimIK}.
In the rest of the proof, $\int\int$ will denote space-time integrals 
over $Q_R:=B_R\times (-R^2,R^2)$. We set 
$$Q_R^1=\bigl\{(x,t)\in Q_R;\ v(x,t)\le 1\bigr\},\quad 
Q_R^2=\bigl\{(x,t)\in Q_R;\ v(x,t)>1\bigr\}.$$
Let us prepare a suitable test-function. We take $\xi\in 
\mathcal{D}(B_1\times (-1,1))$, such that $\xi=1$ in
$B_{1/2}\times (-1/2,1/2)$ and $0\leq \xi\leq 1$.
Due to $m_i>2/3$, we may fix $a$ such that 
$$\max\Bigl\{\frac{m_i+2}{4m_i}, \frac{1}{d_i}, \frac{\gamma_i+1}{2\gamma_i}\Bigr\}<a<1,\quad i\in\{1,2\},$$
with $d_i,\gamma_i>1$ given by Lemma~\ref{lemgrowth}.
By taking $\varphi(x,t)=\varphi_R(x,t)=\xi^b(R^{-1}x,R^{-2}t)$ with $b=b(a)>2$ sufficiently large, we have
\be{S-LiouvBVestimtest}
|\nabla\varphi_R|\leq CR^{-1}\varphi^a, \quad
|\Delta \varphi_R|+\varphi_R^{-1}|\nabla\varphi_R|^2+|\partial_t\varphi_R|\leq CR^{-2}\varphi^a
\ee 
(where $\varphi_R^{-1}|\nabla\varphi_R|^2$ is defined to be $0$ whenever $\varphi=0$).

\medskip

{\bf Step 2.} {\it Energy estimate.}
We shall control the local $L^2$ norm of the time derivative as follows:
\be{S-LiouvBVestimRHSprelim5}
\int\int \varphi\,(v_t)^2
\leq \eta (I+L)+C_\eta R^{-\theta}+C\int\int |g_\eps|^2\varphi,
\ee
for some $\theta>0$ (depending only on $n$ and $f$).

In view of the proof of \eqref{S-LiouvBVestimRHSprelim5}, we first claim that
\be{S-LiouvBVestimRHSprelim}
\int\int |\nabla v|^2
\bigl(|\Delta \varphi|+\varphi^{-1}|\nabla\varphi|^2+|\varphi_t|\bigr)
\leq\eta (I+L)+C_\eta R^{-\theta_1},
\ee
where $\theta_1=1$ if $n=1$ and 
$\theta_1=\frac{4\gamma_1}{\gamma_1-1}-n-2$ 
($>0$ owing to $\gamma_1<p_S$) if $n\ge 2$.
To this end, we first write
\be{S-LiouvBVestimRHSprelim0}
|\nabla v|^2
\bigl(|\Delta \varphi|+\varphi^{-1}|\nabla\varphi|^2 +|\varphi_t|\bigr)
\leq \eta\varphi\, v^{-2}|\nabla v|^4
+C_\eta \varphi^{-1}v^2(|\Delta \varphi|+\varphi^{-1}|\nabla\varphi|^2
+|\varphi_t|\bigr)^2.
\ee
Let $i=1$ and $n\ge 2$, or $i=2$.
It follows from Young's inequality and \eqref{S-LiouvBVestimtest} that, for all $(x,t)\in Q_R^i$, 
$$\begin{aligned}
|\nabla v|^2
\bigl(|\Delta \varphi|&+\varphi^{-1}|\nabla\varphi|^2 +|\varphi_t|\bigr) \\
&\leq \eta \varphi\,  v^{-2}|\nabla v|^4+\eta\varphi\,v^{2\gamma_i}
+C_\eta\bigl[\varphi^{-(\gamma_i+1)/\gamma_i}(|\Delta \varphi|+\varphi^{-1}|\nabla\varphi|^2
+|\varphi_t|\bigr)^2\bigr]^{\gamma_i/(\gamma_i-1)}\\
&\leq \eta \varphi\,  v^{-2}|\nabla v|^4+\eta\varphi\,v^{2\gamma_i}
+C_\eta R^{-4\gamma_i/(\gamma_i-1)}
\end{aligned} $$
hence, by \eqref{growthf1}, 
\be{S-LiouvBVestimRHSprelim1}
|\nabla v|^2
\bigl(|\Delta \varphi|+\varphi^{-1}|\nabla\varphi|^2 +|\varphi_t|\bigr) \\
\le \eta \varphi\,  v^{-2}|\nabla v|^4+\eta\varphi\,f(v)\tilde f(v)
+C_\eta R^{-4\gamma_i/(\gamma_i-1)}. 
\ee
Next, for all $(x,t)\in Q_R^1$, \eqref{S-LiouvBVestimRHSprelim0} also yields
\be{S-LiouvBVestimRHSprelim2}
|\nabla v|^2
\bigl(|\Delta \varphi|+\varphi^{-1}|\nabla\varphi|^2 +|\varphi_t|\bigr)
\leq \eta \varphi\,  v^{-2}|\nabla v|^4+C_\eta R^{-4}
\ee
(using 
{\cred $a>1/2$}). Inequality \eqref{S-LiouvBVestimRHSprelim} then follows from \eqref{S-LiouvBVestimRHSprelim1} if $n\ge 2$ and from \eqref{S-LiouvBVestimRHSprelim2} and \eqref{S-LiouvBVestimRHSprelim1} with $i=2$ if~$n=1$
{\cred (recalling $R>1$).}

{\cred On the other hand,} for $i\in\{1,2\}$ and all $(x,t)\in Q_R^i$, we have
$$\begin{aligned}
{\cred (\tilde F(v) + F(v))(|\varphi_t| + |\Delta\varphi|)\ }
&{= \cred (F(v)+\tilde F(v))}\varphi^{1/d_i}\varphi^{-1/d_i}  (|\varphi_t|{\cred\ + |\Delta\varphi|)}\\
&\le \eta \varphi {\cred (F(v)+\tilde F(v))}^{d_i}+
C_\eta \bigl((|\varphi_t|{\cred\ + |\Delta\varphi|)}\varphi^{-1/d_i}\bigr)^{d_i/(d_i-1)}
\end{aligned} $$
hence, by \eqref{growthf2}, \eqref{S-LiouvBVestimtest},
\be{S-LiouvBVestimRHSprelim2b}
\int \int {\cred (\tilde F(v) + F(v))(|\varphi_t| + |\Delta\varphi|)\ }
\le \eta L+C_\eta R^{n+2-2d_1/(d_1-1)}.
\ee

Let us next prove \eqref{S-LiouvBVestimRHSprelim5}.
We multiply equation \eqref{S-LiouvBVestIJK0}
by $\varphi\,v_t $ and integrate by parts in $x$ and
$t$. Using Young's inequality and  \eqref{S-LiouvBVestimtest}, we get
$$\begin{aligned}
\int\int &\varphi\, (v_t)^2
 =\int\int\varphi\,\partial_t\Big(F(v)-{|\nabla 
v|^2\over 2}\Bigr)-(\nabla\varphi \cdot\nabla v) v_t+g_\eps\varphi v_t\\
&=\int\int\Bigl({|\nabla v|^2\over 2}-F(v)\Bigr)\varphi_t
-(\nabla\varphi \cdot\nabla v) v_t +g_\eps\varphi v_t\\
&\leq\int\int |\nabla v|^2
\Bigl(\frac12|\varphi_t|+|\nabla\varphi|^2\varphi^{-1}\Bigr)+F(v) |\varphi_t|+|g_\eps|^2\varphi
+{1\over 2}\int\int\varphi\,(v_t)^2,
\end{aligned} $$
hence
\be{S-LiouvBVestimRHSprelim2a}
\int\int \varphi\,(v_t)^2
\leq \int\int |\nabla v|^2 \bigl(|\varphi_t|+2|\nabla\varphi|^2\varphi^{-1}\bigr)+2F(v) |\varphi_t|
+2|g_\eps|^2\varphi.
\ee
Letting $\theta=\min\{\theta_1,\frac{2d_1}{d_1-1}-n-2\}>0$ owing to $d_1<(n+2)/n$,
we deduce \eqref{S-LiouvBVestimRHSprelim5} from 
\eqref{S-LiouvBVestimRHSprelim},
\eqref{S-LiouvBVestimRHSprelim2b} and \eqref{S-LiouvBVestimRHSprelim2a}.
\medskip

{\bf Step 3.} {\it Gradient estimate.}
Using Young's inequality, \eqref{Qbddunbddnonc} and \eqref{S-LiouvBVestimtest},
for $i\in\{1,2\}$ and $(x,t)\in Q_R^i$, we have
$$\begin{aligned}
\tilde f(v)|\nabla v\cdot\nabla\varphi|
&\le \varphi^{1/4}v^{-1/2}|\nabla v|\varphi^{-1/4} v^{1/2} \tilde f(v)\,|\nabla\varphi| \\
&\le \eta \varphi v^{-2}|\nabla v|^4+C_\eta\bigl[v^{1/2} \tilde f(v)\,\varphi^{-1/4}|\nabla\varphi|\bigr]^{4/3} \\
&\le \eta \varphi v^{-2}|\nabla v|^4+C_\eta\varphi^{2/(3m_i)}(v^{1/2} \tilde f(v))^{4/3}\,\bigl(\varphi^{-(m_i+2)/(4m_i)}|\nabla\varphi|\bigr)^{4/3} \\
&\le \eta \varphi v^{-2}|\nabla v|^4+\eta \varphi[v^{1/2} \tilde f(v)]^{2m_i}
+C_\eta\bigl(\varphi^{-(m_i+2)/(4m_i)}|\nabla\varphi|\bigr)^{4m_i/(3m_i-2)}\\
&\le \eta C\varphi v^{-2}|\nabla v|^4+\eta C\varphi f(v)\tilde f(v)
+C_\eta R^{-4m_i/(3m_i-2)},
\end{aligned} $$
hence
\be{S-LiouvBVestimRHSprelim2c}
\int\int \tilde f(v)|\nabla v\cdot\nabla\varphi|
\le \eta (I+L)+C_\eta R^{-\theta_2}, 
\ee
with $\theta_2=\frac{4m_1}{3m_1-2}-n-2>0$ by our assumption.
Writing also
$$v^{-1} |\nabla v|^2|\nabla v\cdot\nabla\varphi|
\le v^{-1} |\nabla v|^2(\eta\varphi v^{-1}|\nabla v|^2+C_\eta v\varphi^{-1}|\nabla\varphi|^2)
= \eta\varphi\frac{|\nabla v|^4}{v^2}+C_\eta |\nabla v|^2\frac{|\nabla\varphi|^2}{\varphi}$$
and using \eqref{S-LiouvBVestimRHSprelim5}, 
\eqref{S-LiouvBVestimRHSprelim} (applied with $\tilde\eta=\eta/C_\eta$), we get
$$ 
\begin{aligned}
\int\int & \varphi\,|v_t|\,v^{-1}|\nabla v|^2+\bigl(|v_t|+v^{-1} |\nabla v|^2+|g_\eps|\bigr)|\nabla v\cdot\nabla\varphi|\\
&\leq 2 \eta I 
+C_\eta \int\bigl[\varphi (v_t)^2 +\varphi g_\eps^2+\varphi^{-1}|\nabla v|^2|\nabla\varphi|^2] \\
&\leq 4\eta (I+L)+C_\eta\int\int \varphi g_\eps^2+C_\eta R^{-\theta}.
\end{aligned} 
$$  
Combining this with \eqref{S-LiouvBVestimRHSprelim5}, \eqref{S-LiouvBVestimRHSprelim}, 
and with \eqref{S-LiouvBVestimRHSprelim2b}, \eqref{S-LiouvBVestimRHSprelim2c}, respectively, we obtain
\be{S-LiouvBVestimRHS-T1}
T_1\le \eta (I+L)+C_\eta \int\int \varphi g_\eps^2+C_\eta R^{-\theta}
+C_\eta R^{-\theta_1},
\ee
\be{S-LiouvBVestimRHS-T2}
T_2\le \eta (I+L)+C_\eta \int\int \varphi g_\eps^2+C_\eta R^{-\theta}
+C_\eta R^{-\theta_2}
\ee
(where the $T_i$ are defined in \eqref{S-LiouvBVestimIK}).
Moreover, {\cred we have}
\be{S-LiouvBVestimRHS-T3}
\begin{aligned}
T_3\le \eta{\cred I} 
+C_\eta\int\int \varphi g^2_\eps+\varphi \tilde f(v)|g_\eps|.
\end{aligned} 
\ee

{\bf Step 4.} {\it Conclusion.}
Now, combining 
\eqref{S-LiouvBVestimIK} with \eqref{S-LiouvBVestimRHS-T1}-\eqref{S-LiouvBVestimRHS-T3}
and taking $\eta$ sufficiently small, it follows that
$$
\int\int_{Q_{R/2}} f(u+\eps)\tilde f(u+\eps)\le I+ L \le 
CR^{-\bar\theta}+C\int\int_{Q_R} g^2_\eps+\tilde f(v)|g_\eps|$$ 
(with $\bar\theta=\min(\theta,\theta_1,\theta_2)>0$ and $C>0$ independent of {\cred $\eps\in(0,1)$ and $R>1$}).
For fixed $R{\cred \ >1}$, since $u$ is bounded on $Q_R$ and $f, \tilde f$ are continuous on $[0,\infty)$ 
(recalling \eqref{S-LiouvParabHypInt}),
we have $\sup_{Q_R} |g_\eps|\to 0$ as $\eps\to 0$ 
as well as $\ds\sup_{\eps\in(0,1),\, (x,t)\in Q_R}\tilde f(v)<\infty$.
Letting $\eps\to 0$, we deduce that
$$\int\int_{Q_{R/2}} f(u)\tilde f(u)\le CR^{-\bar\theta}.$$
Letting $R\to \infty$, we conclude that $u\equiv 0$.
\end{proof}

\begin{proof}[Proof of Theorem~\ref{S-thmLiouvJLV2}] 
We shall check that the assumptions of Theorem~\ref{S-thmLiouvJLV1}
are satisfied.

We first note that \eqref{S-LiouvParabHypBV} is a consequence of \eqref{S-LiouvParabHyp1}.
Indeed, if $n\ge 2$, then for all $s>0$,
$$p\tilde f(s)=p\int_0^s z^{p-1}z^{-p}f(z)\,dz
\ge s^{-p}f(s) \int_0^s pz^{p-1}\,dz=f(s).$$

To check \eqref{Qbddunbddnonc}, we first consider the case of assumption \eqref{S-LiouvParabHyp1b}.
Since $f$ is a regularly varying function at $0$ 
 with index $b>1$, as a consequence of \eqref{SlowVarLog},
for all $\eps>0$, we have
\be{S-LiouvBVestimRHSprelim6}
c_\eps s^{b+\eps}\le f(s) \le C_\eps s^{b-\eps},\quad s\in(0,1],
\ee
hence
\be{S-LiouvBVestimRHSprelim7}
\tilde c_\eps s^{b+\eps}\le \tilde f(s) \le \tilde C_\eps s^{b-\eps},\quad s\in(0,1],
\ee
for some constants $c_\eps, C_\eps, \tilde c_\eps, \tilde C_\eps>0$. 
Moreover,  \eqref{S-LiouvParabHyp1} implies $b<p_B<p_S$.
Combining \eqref{S-LiouvBVestimRHSprelim6}-\eqref{S-LiouvBVestimRHSprelim7}, we get
$$f(s)\ge c_\eps s^{b+\eps}\ge c_\eps s^{(2m-1)(b-\eps)+m}\ge \hat c_\eps s^{m} \tilde f^{2m-1}(s),\quad s\in(0,1]$$
that is, \eqref{Qbddunbddnonc} with $i=1$, provided 
$b+\eps\le(2m-1)(b-\eps)+m$,
which is true for $\eps>0$ small whenever $b<(2m-1)b+m$,
i.e.~$m>2b/(2b+1)$.
Since $b<p_S$ implies $2b/(2b+1)<m^*=2(n+2)/(3n+2)$, this is true for $m<m^*$ close to $m^*$.
Similarly, since $f$ is a regularly varying function at $\infty$ with index $\bar b\in (1,p_c)$,
we get \eqref{Qbddunbddnonc} with $i=2$ by taking $\bar m\in(2/3,m)$ close to $2/3$.

Next consider the case of assumption \eqref{S-LiouvParabHyp3}.
There exist $s_2\ge s_1>0$ such that
$$\tilde f(s)=\int_0^s z^{\sigma-1}z^{-\sigma}f(z)\,dz
\le s^{-\sigma}f(s) \int_0^s z^{\sigma-1}\,dz=\sigma^{-1}f(s),\quad s\in (0,s_1),$$
and
$$
\begin{aligned}
\tilde f(s)
&=\int_0^{s_2} z^{-1}f(z)\,dz+\int_{s_2}^s z^{-1}f(z)\,dz \\
&\le C+f(s) \int_{s_2}^s z^{-1}\,dz=C+f(s)\log s\le C'f(s)\log s,\quad s\ge s_2.
\end{aligned}
$$
To show \eqref{Qbddunbddnonc}, it thus suffices to check that
\be{compfftilde2}
f^{1-m_1}(s)\ge cs^{m_1/2}\quad\hbox{ for all $s\in (0,s_1)$}, 
\ee
and
\be{compfftilde2b}
f^{1-m_2}(s)\ge cs^{m_2/2}(\log s)^{m_2-\frac12}\quad\hbox{ for all $s\ge s_2$}, 
\ee
with some ${2\over 3}<m_2<m_1<m^*$.
Taking $m_2\in(2/3,1)$ close to $2/3$,
inequality~\eqref{compfftilde2b} follows from assumption~\eqref{S-LiouvParabHyp3}.
For $n=1$, we may take $m_1=1$, hence \eqref{compfftilde2} is true.
Finally let $n\ge 2$. By \eqref{S-LiouvParabHyp1}, we then have $f(s)\ge cs^p$ for all $s\in\omega_1$, where $p<p_B<p_S$.
This guarantees \eqref{compfftilde2} for $m_1<m^*\le 1$ close to $m^*$.
\end{proof}

\section{Technical lemmas} \label{SecProofTech}
In this section we gather a number of technical results that are required
for the sake of the examples or counter-examples.

\smallskip

We start with Example 1 of Section~\ref{SubSecEx}.
In cases (i) (resp., (iii)) it is not difficult to check that $f$ satisfies {\cred assumption} \eqref{hypUB1b} 
(resp., \eqref{hypUB1c}).
In case (ii), the second part of assumption \eqref{hypUB1a} is easy to check, 
and the first part follows from the next lemma.

\begin{lem} \label{lemEx1}
Let $p\in(1,p_c)$. Under assumptions (ii) of Example 1, the function $s^{-m}f(s)$ is nonincreasing on $(0,\infty)$ for $m<p_c$ close to $p_c$.
\end{lem}

\begin{proof}
Let $g(s)=s^{-m}f(s)=s^{p-m}\log^q(K+s)$. It suffices to show that $g'(s)\le 0$ on $(0,\infty)$ for some $m\in(p,p_c)$. 
We compute
$$g'(s)=(p-m)s^{p-m-1}\log^q(K+s)+qs^{p-m}(K+s)^{-1}\log^{q-1}(K+s),$$
which has the sign of
$h(s):=(p-m)(K+s)\log(K+s)+qs$. Note that $h(0)\le 0$. We have
$$h'(s)=(p-m)[1+\log(K+s)]+q\le \delta:=(p-m)(1+\log K)+q.$$
If $\delta_0:=(p-p_c)(1+\log K)+q<0$, then $\delta\le 0$ for $m<p_c$ close to $p_c$,
hence $g'\le 0$ on $(0,\infty)$.
This is always true under assumption (ii1) or (ii2). Thus
 assume (ii3) and $\delta_0\ge 0$, hence $\delta>0$. 
The function $h'$ has unique zero $s_0>0$, given by $\log(K+s_0)=\frac{q}{m-p}-1$. 
Since $h<0$ at $0$ and $\infty$, the global maximum of $h$ is attained at $s_0$.
We then compute 
$$\begin{aligned}
h(s_0)
&=(p-m)(K+s_0)\log(K+s_0)+qs_0
=-(K+s_0)(q+p-m)+qs_0 \\
&=(K+s_0)(m-p)-qK 
=(m-p)\exp\bigl[\ts\frac{q}{m-p}-1\bigr]-qK<0
\end{aligned}$$
for $m$ close to $p_c$, hence again $g'\le 0$ on $(0,\infty)$.
\end{proof}

The last part of Example~1 is based on:

\begin{prop} \label{lemEx1b}
Under assumption \eqref{Ex1a2}
there exists a radially symmetric decreasing, positive classical solution of $-\Delta v=v^p\log^q(K+v)$ in $\R^n$.
\end{prop}

\begin{proof}
We modify an argument from \cite{QS12} (see \cite[Theorem~1.4]{QS12}).
Let $F(s)=\int_0^s f(z)dz$,
$\phi(s)=sf(s)-(p_S+1)F(s)$ and note that $\phi(0)=0$.
We compute
$$\phi'(s)=(p-p_S)s^p\log^q(K+s)+qs^{p+1}(K+s)^{-1}\log^{q-1}(K+s),$$
which has the sign of $h(s):=(p-p_S)(K+s)\log(K+s)+qs$.
Therefore, $\phi\ge 0$ on $[0,\infty)$ when $p\ge p_S$, $q>0$ and $K\ge 1$.
Next consider the case $1<p<p_S$, $q>p_S-p$ and $K=1$. 
Then we have $h(0)=0$ and $h'(0)=p-p_S+q>0$, hence
$h'>0$ on $[0,s_0]$ for $s_0>0$ small.
In all cases we thus have
\be{signphi}
\phi(s)=sf(s)-(p_S+1)F(s)\ge 0\quad\hbox{for all $s\in [0,s_0]$}.
\ee

Now extend $f$ by $0$ for $s<0$ and consider the initial value problem
$$-(r^{n-1}v')'= r^{n-1}f(v),\  r > 0, \ \quad v(0) = s_0, \quad v'(0) = 0.$$
Let $R^*$ be its maximal existence time. Since $f\ge 0$, we have $v'\le 0$ on $[0,R^*)$. 
Assume for contradiction that $v$ has a (first) zero $R\in (0,R^*)$. Then $v$ is a classical solution of 
$-\Delta v=f(v)$ on the ball $B_R$ with Dirichlet boundary conditions and $0\le v\le s_0$.
But in view of \eqref{signphi} and Pohozaev's inequality (see, e.g.,~\cite[Corollary 5.2]{QSb}{\cb )}, this is a contradiction.
We conclude that $v>0$ on $(0,\infty)$, which proves the proposition.
\end{proof}

In Example 2 of Section~\ref{SubSecEx} we used the following:

\begin{lem} \label{lemEx2b}
For $p>1$ and $a\in(0,p-1)$, the function 
$$f(s)=s^{p+a\sin[\log\log(3+s+s^{-1})]}$$
{\cb has} controled variation.
\end{lem}

\begin{proof}
It suffices to show the boundedness of $\zeta=sf'/f$ {\cb on $(0,\infty)$}
(cf.~after \eqref{ControlVarLog}).
Writing $f(s)=e^\phi$ with $\phi(s)=[p+a\sin(h(s))] \log s$ and $h(s)=\log\log(3+s+s^{-1})$, we have
$$\zeta=s\phi'= p+a\sin(h(s))+ a s h'(s) \cos [h(s)] \log s$$
and it thus suffices to verify the boundedness on $(0,\infty)$ of 
$$s h'(s)\log s
={s\bigl\{ \log(3+s+s^{-1})\}' \log s\over \log(3+s+s^{-1})}
={(s-s^{-1})\log s\over (3+s+s^{-1})\log(3+s+s^{-1})},$$
which is true.
\end{proof}

In Example 4 of Section~\ref{SubSecEx} we used the following:

\begin{lem} \label{lemEx4}
For any $1<\ell<m<p<p_*$, there exist $f$ satisfying \eqref{hypcontpos}, \eqref{Ex1} and 
$$\hbox{$f(s_i)=s_i^p$, $f(t_i)=t_i^m$ for some sequences $s_i,t_i\to\infty$.}$$
\end{lem}

\begin{proof}
Fix $\bar m\in (\ell,m)$, $\bar p \in (p,p_*)$ and let
$$f(s)=\chi_{[0,2)}s^p+
\sum_{i=1}^\infty \chi_{[s_{2i},s_{2i+1})} s_{2i}^{p-\bar m}s^{\bar m}
+\chi_{[s_{2i+1},s_{2i+2})} s_{2i+1}^{m-\bar p}s^{\bar p}$$
where $s_2=2$ and $s_i$ is defined inductively by $s_{2i+1}=s_{2i}^{(p-\bar m)/(m-\bar m)}$ and 
$s_{2i+2}=s_{2i+1}^{(\bar p-{\cred m})/(\bar p-{\cred p})}$. 
We have $f(s_{2i}^+)=f(s_{2i})= s_{2i}^p$ and $f(s_{2i+1}^+)=f(s_{2i+1})= s_{2i+1}^m$.
The choice of $s_i$ guarantees
$f({s_{2i+1}}^-)=s_{2i}^{p-\bar m}s_{2i+1}^{\bar m}=s_{2i+1}^m$ and
$f(s_{2i+2}^-)=s_{2i+1}^{m-\bar p}s_{2i+2}^{\bar p}=s_{2i+2}^p$,
hence $f$ is continuous.
Moreover, 
$$sf'f^{-1}=
\begin{cases}
p&\hbox{ on $[0,s_2)$} \\
\noalign{\vskip 1mm}
\bar m,&\hbox{ on $(s_{2i},s_{2i+1})$} \\
\noalign{\vskip 1mm}
\bar p,&\hbox{ on $(s_{2i+1},s_{2i+2})$.}
\end{cases}
$$
It follows that $s^{-\ell}f(s)$ is increasing and $s^{-p_*}f(s)$ is decreasing. 
\end{proof}

The counter-example in Remark~\ref{remLiouvParab}(iii) is based on the following:

\begin{lem} \label{lemEx2}
Let $f, \tilde f$ be defined in \eqref{S-LiouvParabHypInt}, \eqref{counterexftilde},
with $1<m<p_B<q$ and $a=\bigl[\frac{p_B-m}{q-p_B}\bigr]^{\frac{1}{q-m}}$.

(i) There exists $p\in (1,p_B)$ such that $s^{-p}\tilde f(s)$ is nonincreasing on $(0,\infty)$.

(ii) For any $p\in (1,p_B)$, the function $s^{-p}f(s)$ is not nonincreasing on $(0,\infty)$.
\end{lem}

\begin{proof}
We compute
$$\tilde f(s)=\begin{cases}
\frac{1}{m}s^m+\frac{1}{q}s^q,&\hbox{ if $s\in[0,a]$} \\
\noalign{\vskip 1mm}
\frac{1}{m}(1+a^{q-m})s^m+(\frac{1}{q}-\frac{1}{m})a^q,&\hbox{ if $s\ge a$,}
\end{cases}
$$
$$sf'(s)=\begin{cases}
ms^m+qs^q,&\hbox{ if $s\in[0,a)$} \\
\noalign{\vskip 1mm}
m(1+a^{q-m})s^m,&\hbox{ if $s>a$,}
\end{cases}
\qquad
s\tilde f'(s)=\begin{cases}
s^m+s^q,&\hbox{ if $s\in[0,a)$} \\
\noalign{\vskip 1mm}
(1+a^{q-m})s^m,&\hbox{ if $s>a$.}
\end{cases}
$$
For any $p\in (1,p_B)$, setting $s_p=\bigl[\frac{(p-m)_+}{q-p}\bigr]^{1/(q-m)}<a$, we have
$$sf'-pf=ms^m+qs^q-p(s^m+s^q)=(m-p)s^m+(q-p)s^q>0,\quad s\in(s_p,a),$$ 
which proves assertion (ii). On the other hand, choose $p\in(m,p_B)$ close to $p_B$, so that 
$\bar s_p:=\bigl[\frac{q(p-m)}{m(q-p)}]^{1/(q-m)}>a$. Then we have
$$s\tilde f'-p\tilde f=s^m+s^q-p(\ts\frac{1}{m}s^m+\frac{1}{q}s^q)
=(1-\frac{p}{m})s^m+(1-\frac{p}{q})s^q\le 0,\quad s\in(0,a),$$
whereas
$$\begin{aligned}
s\tilde f'-p\tilde f
&=\ts(1-\frac{p}{m})(1+a^{q-m})s^m+p(\frac{1}{m}-\frac{1}{q})a^q
\le \ts(1-\frac{p}{m})(a^m+a^q)+p(\frac{1}{m}-\frac{1}{q})a^q \\
&{\cred =\,}\ts(1-\frac{p}{m})a^m+(1-\frac{p}{q})a^q\le 0, \quad s>a. 
\end{aligned}$$
Assertion (i) follows. 
\end{proof}

\medskip

Finally, we justify the assertion in Remark~\ref{formHardy0}. 

\begin{prop} \label{propSharp}
Let $f$ be as in Theorem~\ref{thm1-0} or \ref{thm1}. Then all solutions
of the ODE $y'=f(y)$ with $y(0)>0$ blow up in finite time $T$ and satisfy
\be{lowerBU}
\frac{f(y(t))}{y(t)}\ge \frac{c}{T-t},\quad t\to T,
\ee
for some constant $c>0$ depending only on $f$.
\end{prop}

\begin{proof}
Since $f>0$ on $(0,\infty)$, if $y(0)>0$ then $y$ is increasing and cannot stay bounded. Also, by Lemma~\ref{lemRegVar}, there exist
 $m>1$ and $c>0$ such that $f(s)\ge cs^m$ for all $s>1$.
 Consequently $T$ is finite, with $\lim_{t\to T}y(t)=\infty$,
and $y$ is given by $H(y(t))=T-t$, where $H(s)=\int_s^\infty \frac{dz}{f(z)}\ <\infty$.

On the other hand, the assumptions of Theorem~\ref{thm1-0} or \ref{thm1}
imply that $f$ has regular variation at $\infty$ or controled variation. 
Therefore (see before \eqref{ControlVarLog}), there exist functions $\tau,\xi$ and,
and constants $C_1,C_2>1$ such that 
$$f(s)=\tau(s)f_0(s)\ \ \hbox{ for all $s\ge 1$,\quad where }
f_0(s)=\exp\Bigl[\int_1^s z^{-1}\xi(z)\,dz\Bigr],$$
and $|\xi(s)|\le C_2$, $C_1\le \tau(s)\le C_2$ on $[1,\infty)$.
Take $q>C_2>1$. We see that
$s^{-q}f_0(s)=\exp\bigl[\int_1^s z^{-1}[\xi(z)-q]\,dz\bigr]$ is decreasing for $s\ge 1$.
Consequently, for $s\ge 1$, we have
$$\begin{aligned}
H(s)
&\ge C_2^{-1}\int_s^\infty \frac{dz}{f_0(z)}
=C_2^{-1}\int_s^\infty z^{-q} \frac{z^q}{f_0(z)}dz\\
&\ge C_2^{-1}\frac{s^q}{f_0(s)} \int_s^\infty z^{-q}dz=(C_2(q-1))^{-1}\frac{s}{f_0(s)} \ge c\frac{s}{f(s)}
\end{aligned}$$
with $c=C_1(C_2(q-1))^{-1}$. Since $y(t)>1$ for $t$ close enough to $T$, property \eqref{lowerBU} follows.
\end{proof}

{\bf Acknowledgements.} The author thanks J.~Garc\'\i a-Meli\'an for stimulating discussion during the preparation of this work.
{\cb He also thanks the referee for careful reading of the manuscript and useful suggestions.}


\begin{thebibliography}{99}

\bibitem{BCDN} H. Berestycki, I. Capuzzo-Dolcetta, L. Nirenberg, 
Superlinear indefinite elliptic problems and nonlinear Liouville theorems, 
{\cred Topol.} Methods Nonlinear Anal. 4 (1994), 59--78.

\bibitem{Ber} S.N. Bernstein, Complete Works, Vol II. Constructive function theory [1931-1953] (Russian), Izdat. Akad. Nauk SSSR, Moscow, (1954). 

\bibitem{BV98}M.-F. Bidaut-V\'eron,
 Initial blow-up for the solutions of a semilinear parabolic equation with source term, 
 Equations aux d\'eriv\'ees partielles et applications, articles d\'edi\'es \`a Jacques-Louis Lions, Gauthier-Villars, Paris, 1998, pp. 189-198.


\bibitem{BVV}
M.-F. Bidaut-V\'eron, L. V\'eron, 
Nonlinear elliptic equations on compact Riemannian manifolds and asymptotics of Emden equations, 
Invent. Math. 106 (1991), 489--539.

\bibitem{BVY}
M.-F. Bidaut-V\'eron, C. Yarur, 
Semilinear elliptic equations and systems with measure data: existence and a priori estimates, 
Adv. Differential {\cred Equations} 7 (2002), 257--296.

\bibitem{BGT}
N.H. Bingham, C.M. Goldie, J.L. Teugels, 
Regular Variation, Encyclopedia Math. Appl., vol. 27, 
Cambridge University Press, Cambridge, 1987.

\bibitem{BT}
H. Brezis, R.E.L. Turner, 
On a class of superlinear elliptic problems, Comm. Partial Differential Equations 2 (1977), 601--614.

\bibitem{CLi}
W. Chen, C. Li, 
Classification of solutions of some nonlinear elliptic equations, 
Duke Math. J. 63 (1991), 615--622.


\bibitem{CF}
M.~Chleb\'{\i}k, M.~Fila,
From critical exponents to blowup rates for parabolic problems,
Rend. Mat. Appl., Ser. VII 19 (1999), 449--470.

\bibitem{Cir}
F. Cirstea,
A complete classification of the isolated singularities for nonlinear elliptic equations with inverse square potentials,
Mem. Amer. Math. Soc. 227 (2014), 1068.

\bibitem{CirDu}
F. Cirstea, Y. Du,
Asymptotic behavior of solutions of semilinear elliptic equations near an isolated singularity,
J. Funct. Anal. 250 (2007), 317-346.

\bibitem{CRT}
D. Costa, H. Ramos Quoirin, H. Tehrani, 
A variational approach to superlinear semipositone elliptic problems,
Proc. Am. Math. Soc. 145 (2017), 2661--2675.

\bibitem{Dan}
E.N. Dancer, 
Some notes on the method of moving planes, 
Bull. Austral. Math. Soc. 46 (1992), 425--434.

\bibitem{DNZ}
G.K. Duong, V.T. Nguyen, H. Zaag, 
Construction of a stable blowup solution with a prescribed behavior for a non-scaling-invariant semilinear heat equation,
Tunis. J. Math. 1 (2019), 13--45.

\bibitem{FS} 
M. Fila, Ph.~Souplet,
The blowup rate for semilinear parabolic problems on general domains,
NoDEA Nonlinear Differ. Equations Appl. 8 (2001), 473--480.


\bibitem{FML85}	         
 A. Friedman, B. McLeod,  
Blow-up of positive solutions of semilinear heat equations,        
Indiana Univ. Math. J. 34 (1985),  425--447.

\bibitem{Fuj}
 Y. Fujishima, 
 Blow-up set for a superlinear heat equation and pointedness of the initial data, 
 Discrete Continuous Dynamical Systems A 34 (2014), 4617--4645.
 
\bibitem{FIM}
Y. Fujishima, K. Ishige, H. Maekawa, 
Blow-up set of type I blowing up solutions for nonlinear parabolic systems, 
Math. Ann. 369 (2017), 1491--1525.

\bibitem{GIR}
J. Garc\'\i a-Meli\'an, L. Iturriaga, H. Ramos Quoirin,
A priori bounds and existence of solutions for slightly superlinear elliptic problems,
Adv. Nonlinear Stud. 15 (2015), 923--938.

\bibitem{GSa} 
B.~Gidas, J.~Spruck,
Global and local behavior of positive solutions of nonlinear elliptic equations,
Commun. Pure Appl. Math. 34 (1981), 525--598.

\bibitem{GSb} 
B.~Gidas, J.~Spruck,
A~priori bounds for positive solutions of a nonlinear elliptic equations,
Commun. Partial. Differ. Equations. 6 (1981), 883--901.


\bibitem{Gi} 
Y.~Giga, A bound for global solutions of semilinear heat equations,
Comm. Math. Phys. 103 (1986), 415--421.

\bibitem{GK} 
Y.~Giga, R.~Kohn,
Characterizing blowup using similarity variables,
Indiana Univ. Math. J. 36 (1987), 1--40.

\bibitem{HZ} M.A. Hamza, H. Zaag,
The blow-up rate for a non-scaling invariant semilinear wave equations in higher dimensions,
Nonlinear Anal. 212 (2021), 112445.

\bibitem{HZ2} M.A. Hamza, H. Zaag,
The blow-up rate for a non-scaling invariant semilinear heat equation,
Arch. Rat. Mech. Anal. {\cb 244 (2022), 87-125.}

\bibitem{Hu} 
B. Hu,
Remarks on the blowup estimate for solutions of the heat
equation with a nonlinear boundary condition,
Differ. Integral Equations 9 (1996), 891--901.

\bibitem{JWY} 
A. Jevnikar, J. Wang, W. Yang
Liouville type theorems and periodic solutions for $\chi^{(2)}$ type systems with non-homogeneous nonlinearities,
ArXiv:2008.13190.

\bibitem{Ka} J. Karamata,
Sur un mode de croissance r\'eguli\`ere des functions (French),
Mathematica 4 (1930), 38--53.

\bibitem{Ka2} J. Karamata,
Bemerkung \"uber die vorstehende Arbeit des Herrn Avakumovi\'c, mit n\"aherer Betrachtung einer Klasse von Funktionen, 
welche bei den Inversionss\"atzen vorkommen,
Bull. intern. Acad. Yougoslave Sci., Cl. Sci. math. nat. 29-30 (1936), 117-123.

\bibitem{Kur}
V.V. Kurta, 
A Liouville comparison principle for solutions of semilinear parabolic inequalities in the whole space, 
Adv. Nonlinear Anal. 3 (2014), 125--131.


\bibitem{LZ}
Y. Li, L. Zhang,
Liouville-type theorems and Harnack-type inequalities for semilinear elliptic equations, 
J. Anal. Math. 90 (2003), 27--87.


\bibitem{MS}J. Matos, Ph. Souplet, Universal blow-up rates for a semilinear heat equation and applications,
Advances Differ. Equations, 8 (2003), 615--639.

\bibitem{MP}
E. Mitidieri, S.I. Pohozaev, 
A priori estimates and blow-up of solutions of nonlinear partial differential equations and inequalities, 
Proc. Steklov Inst. Math. 234 (2001), 1--362.

\bibitem{NNPY}
Q.A Ngo, V.H. Nguyen, Q.H. Phan, D. Ye,
Exhaustive existence and non-existence results for some prototype polyharmonic equations in the whole space,
J. Differ. Equations 269, (2020) 11621--11645.



\bibitem{P20}P. Pol\'a\v cik, 
Entire solutions and a Liouville theorem for a class of parabolic equations on the real line,
Proc. Am. Math. Soc. 148 (2020), 2997--3008.



\bibitem{PQ}P. Pol\'a\v cik, P. Quittner, 
A Liouville-type theorem and the decay of radial solutions of a semilinear heat equation, 
Nonlinear Anal. 64 (2006), 1679--1689.


\bibitem{PQS1}
P. Pol\'a\v cik, P. Quittner, Ph. Souplet
Singularity and decay estimates in superlinear problems
            via Liouville-type theorems.  Part I: elliptic equations and systems,
Duke Math. J.
139 (2007), 555--579.

\bibitem{PQS2}
P. Pol\'a\v cik, P. Quittner, Ph. Souplet
Singularity and decay estimates in superlinear problems
            via Liouville-type theorems.  Part II: parabolic equations,
Indiana Univ. Math. J.
56 (2007), 879--908.


\bibitem{Q16}	    P. Quittner,
Liouville theorems for scaling invariant superlinear parabolic problems with gradient
structure,
Math. Ann. 364 (2016), 269--292.

\bibitem{Q21}  P. Quittner, Optimal Liouville theorems for superlinear parabolic problems. Duke Math. J. 170 (2021), 1113--1136.

\bibitem{Q21b}  P. Quittner, Liouville theorems for parabolic systems with homogeneous nonlinearities and gradient structure,
{\cb Partial Differ. Equations Appl. 3 (2022), 26.}

\bibitem{QS04}	          P. Quittner, Ph. Souplet,
A priori estimates and existence for elliptic systems via bootstrap in weighted Lebesgue spaces, 
Arch. Rational Mech. Anal. 174 (2004), 49--81.


\bibitem{QS12}	          P. Quittner, Ph. Souplet,
Symmetry of components for semilinear elliptic systems,
SIAM J. Math. Anal. 44 (2012), 2545--2559.

\bibitem{QSb}	          P. Quittner, Ph. Souplet,
Superlinear parabolic problems. Blow-up, global existence and steady states,
Second Edition, Birkh\"auser Advanced Texts, 2019.

\bibitem{RZ}	 
W. Reichel, H. Zou, 
Non-existence results for semilinear cooperative elliptic systems via moving spheres, 
J. Differential Equations 161 (2000), 219--243.

\bibitem{Se}
E. Seneta, 
Regularly Varying Functions, Lecture Notes in Math., vol. 508, Springer-Verlag, Berlin, New York, 1976.

\bibitem{SZ96}
J. Serrin, H. Zou,
Non-existence of positive solutions of Lane-Emden systems, 
Differential Integral Equations 9 (1996), 635--653.

\bibitem{SZ}
J. Serrin, H. Zou,
Cauchy-Liouville and universal boundedness theorems for quasilinear
elliptic equations and inequalities, Acta Math. 189  (2002), 79--142.

\bibitem{Sir}
B. Sirakov,
A new method of proving a priori bounds for superlinear elliptic PDE,
J. Math. Pures Appl. (9) 141 (2020), 184--194.


\bibitem{S09}
Ph. Souplet, 
The proof of the Lane-Emden conjecture in four space dimensions, 
Adv. Math. 221 (2009), 1409--1427.

\bibitem{S12}
Ph. Souplet, 
Liouville-type theorems for elliptic Schr\"odinger systems associated with copositive matrices, 
Netw. Heterog. Media 7 (2012), 967--988.


\bibitem{Tal0}
S.D. Taliaferro, 
 Isolated singularities of nonlinear elliptic inequalities, 
 Indiana Univ. Math. J. 50 (2001) 1885--1897.

\bibitem{Tal}
S.D. Taliaferro, 
Local behavior and global existence of positive solutions of 
$au^\lambda \le -\Delta u\le u^\lambda$, 
Ann. Inst. H. Poincar\'e Anal. Non Lin\'eaire 19 (2002), 889--901.

\bibitem{Tal09}
S.D. Taliaferro, 
Blow-up of solutions of nonlinear parabolic inequalities, 
Trans. Amer. Math. Soc. 361 (2009), 3289--3302.

\bibitem{WY}
J. Wei, D. Ye, 
Liouville theorems for stable solutions of biharmonic problem, 
Math. Ann. 356 (2013), 1599--1612.
  
\end{thebibliography}
\end{document}